\documentclass{article}

\usepackage[english]{babel}
\usepackage{amssymb}
\usepackage{amsfonts}
\usepackage{amsmath}
\usepackage{amsthm}
\usepackage{graphicx}
\usepackage{caption}
\usepackage{float}
\usepackage{verbatim}

\newtheorem{theorem}{Theorem}[section]
\newtheorem{corollary}[theorem]{Corollary}
\newtheorem{definition}[theorem]{Definition}
\newtheorem{notation}[theorem]{Notation}
\newtheorem{proposition}[theorem]{Proposition}
\newtheorem{remark}[theorem]{Remark}

\numberwithin{equation}{section}

\begin{document} 

\title{On some concepts of convergence and their connections}

\author{Tatiana Danielsson, 
        Liselott Flodén,
        Anders Holmbom,\\
        Pernilla Johnsen,
        Marianne Olsson Lindberg,
        Jens Persson}

{\let\thefootnote\relax\footnotetext{We acknowledge with gratitude the valuable contributions by Professor Anders Holmbom who sadly passed away before the completion of this paper.}}

\date{}
\maketitle

\begin{abstract} 
    The notions of classical two-scale convergence and the periodic unfolding method, which both rely on an additional scale for the representation of periodic micro-oscillations, are in this paper traced back to the wider concept of general two-scale convergence. Moreover, a comparison is made between very weak two-scale convergence and periodic unfolding with respect to sequences bounded in Sobolev spaces, amplified by the scale of the micro-oscillations. It turns out that, for amplified sequences, these approaches are not equivalent in an immediate sense. The paper both serves as a literature overview, with examples and illustrations of the concepts of convergence, and contains new results.
\end{abstract}


\vspace{1.2mm}

\section{Introduction}

The primary objective of this paper is to offer further advances regarding how some concepts of convergence utilized in homogenization theory relate to one another. We make a comparison between classical two-scale convergence and the periodic unfolding method through the wider concept of general two-scale convergence. A result concerning the relation between very weak two-scale convergence and its counterpart in unfolding is also provided. 

It is a well established fact that classical two-scale convergence and
periodic unfolding are equivalent; see, e.g., Proposition~2.14 in \cite{CiDaGr1}. In this paper, we will show a connection between general two-scale
convergence and the adjoint operator of the unfolding operator. More
precisely, we will show that for sequences bounded in $L^{2}(\Omega )$
unfolding is a special case of general two-scale convergence providing us
with a new link between two-scale convergence and unfolding via general
two-scale convergence. 

Sequences of the type $\left\{ \varepsilon_{h}^{-1} u_{\varepsilon _{h}}\right\} $, arising in, e.g., homogenization of parabolic problems with rapid temporal oscillations, can, while being beyond the scope of classical two-scale convergence, be studied in the sense of very weak two-scale convergence. This notion uses a smaller class of test functions than two-scale convergence to filter away the global tendency. Within the framework of periodic unfolding, the effect of the slow global change of such sequences is handled by subtraction with a so-called local average operator. As we will see, even though showing similarities, 
unfolding and very weak two-scale convergence are not interchangeable in the same
manner as unfolding and two-scale convergence are.

The paper is organized as follows: In Sections~\ref{Two-scale}, \ref{Unfolding}
and \ref{Operator governed} we recall the concepts of two-scale convergence along
with the notion of very weak two-scale convergence, unfolding and general two-scale
convergence, respectively. In Section~\ref{Unification} the relations between the different
types of convergences are discussed and the new results mentioned
above are presented. 

\begin{remark}
In this paper we have restricted our discussion to the $L^{2}(\Omega )$-case
but the fundamental theory concerning two-scale convergence, general
two-scale convergence and unfolding has also been developed for $%
L^{p}(\Omega )$, $1<p<\infty $; see, e.g., \cite{LNW}, \cite{HOL1} and \cite%
{CiDaGr2}, respectively. Two-scale convergence, very weak two-scale
convergence and unfolding have been subject to a number of generalizations
that have further broadened the applicability to include, e.g., problems
with more than two scales, problems defined on perforated domains and
evolution problems; see, e.g., \cite{AlBr}, \cite{All1}, \cite{HOL}, \cite{FHOLP33}, \cite{PerPhD}, \cite{FHOLParbitrary}, \cite{Lob}, \cite{CiDaGr2}, 
\cite{CiDoZa}, \cite{CiDaDoGr}, \cite{AAB}, \cite{AAGT} and \cite{CiDaGr4}.
Beyond the scope of this paper we also find, e.g., the notion of $\Sigma$-convergence introduced by Nguetseng; see, e.g., \cite{Ng2}, \cite{NgWo1}
and \cite{NgSv}. Well known types of homogenization problems such as
periodic and almost periodic ones can be treated by techniques that appear
as special cases of $\Sigma $-convergence.
\end{remark}

\begin{notation}
We let $\Omega \subset \mathbb{R}^{N}$ be an open bounded set with smooth boundary, $Y$ be the unit cube $\left( 0,1\right) ^{N}$ and $\varepsilon _{h}$ be a scale function, i.e., a positive function such that $\varepsilon _{h}\rightarrow 0$ as $h\rightarrow \infty $
and let $\varepsilon $ be the size of a repeated cell in $\Omega $, i.e.,
there is no sequence implied. Furthermore, we let $F_{\sharp }(Y)$ be the
space of all functions in $F_{\operatorname{loc}}(\mathbb{R}^{N})$ which are the periodic repetition of some function in the function
space $F(Y)$ and by $F(Y)/\mathbb{R}$ we mean functions $u\in F(Y)$ with integral mean zero.
\end{notation}

\section{Two-scale convergence\label{Two-scale}}

We give some, for our aim, essential properties of two related concepts: Two-scale convergence (here called
classical two-scale convergence or simply two-scale convergence when there
is no ambiguity) and very weak two-scale convergence.

\subsection{Classical two-scale convergence}

The method of two-scale convergence was introduced by Nguetseng in 1989 in 
\cite{Ng1} as a tool for homogenization of partial differential equations and further developed by Allaire in \cite{All1}. Other works that also studied the concept are, e.g., \cite%
{Zhi}, \cite{LNW} and \cite{FHOLP2} to mention a few. Two-scale convergence relies on the choice of test functions in the sense that the effect of rapid oscillations is absorbed by a second scale. 
We give the formal definition.

\begin{definition}
A sequence $\left\{ u_{\varepsilon _{h}}\right\} $ in $L^{2}(\Omega )$ is
said to two-scale converge to\newline
$u_{0}\in L^{2}(\Omega \times Y)$ if%
\begin{equation*}
\lim_{h\rightarrow \infty }\int_{\Omega }u_{\varepsilon _{h}}\left( x\right)
v\left( x,\frac{x}{\varepsilon _{h}}\right) dx=\int_{\Omega
}\int_{Y}u_{0}\left( x,y\right) v\left( x,y\right) dydx
\end{equation*}
for all $v\in L^{2}(\Omega ;C_{\sharp }(Y))$. This is denoted%
\begin{equation*}
u_{\varepsilon _{h}}\left( x\right) \overset{2}{\rightharpoonup }u_{0}\left(
x,y\right) \text{.}
\end{equation*}
\end{definition}

We have the following compactness result.

\begin{theorem}
\label{Th twoskalecompact}Let $\left\{ u_{\varepsilon _{h}}\right\} $ be a
bounded sequence in $L^{2}(\Omega )$. Then, up to a subsequence,%
\begin{equation*}
u_{\varepsilon _{h}}\left( x\right) \overset{2}{\rightharpoonup }u_{0}\left(
x,y\right)
\end{equation*}%
for some $u_{0}\in L^{2}(\Omega \times Y)$.
\end{theorem}

\begin{proof}
See, e.g., Theorem 7 in \cite{LNW}.
\end{proof}

The theorem below concerns the uniqueness of the two-scale limit.
\begin{theorem}
    The two-scale limit is unique.
\end{theorem}
\begin{proof}
    See, e.g., the discussion following Definition~1 in \cite{LNW}.
\end{proof}

Two-scale convergence is related to weak convergence but an essential distinction is that the two-scale limit is able to capture information from microscopic oscillations, see Figures~\ref{Figur: two-scale} and~\ref{Figur: two-scale2}. Note that the sequences have the same weak $L^2(\Omega)$-limits but different two-scale limits.
\begin{figure}[H]
    \centering
    \includegraphics[width=0.32\linewidth]{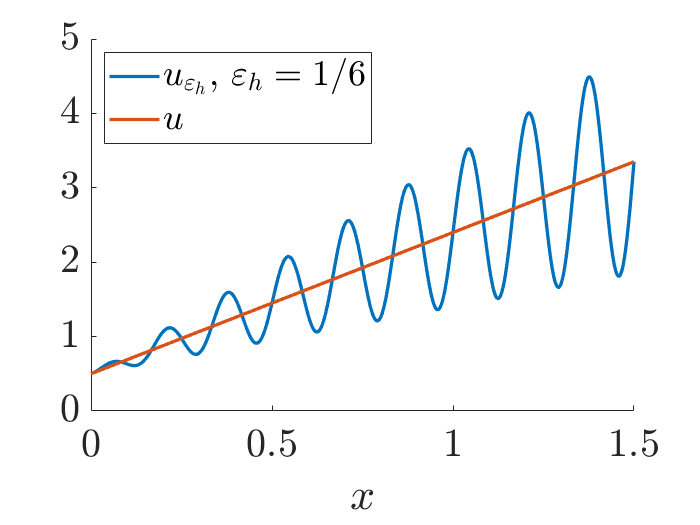}
    \includegraphics[width=0.32\linewidth]{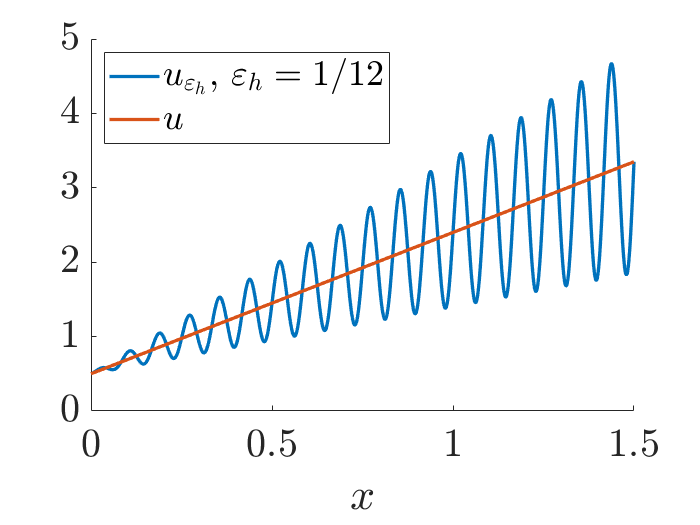}
    \includegraphics[width=0.32\linewidth]{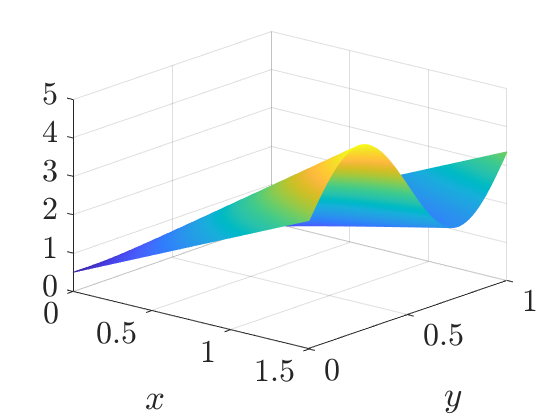}
    \caption{Elements from a sequence $\left\{ u_{\varepsilon _{h}}\right\} $ for two fixed
    ${\varepsilon}_h$  together with the weak 
    $L^2(\Omega)$-limit $u$ (the two leftmost pictures) and the two-scale limit (to the right).}\label{Figur: two-scale}
\end{figure}

\begin{figure}[H]
    \centering
    \includegraphics[width=0.32\linewidth]{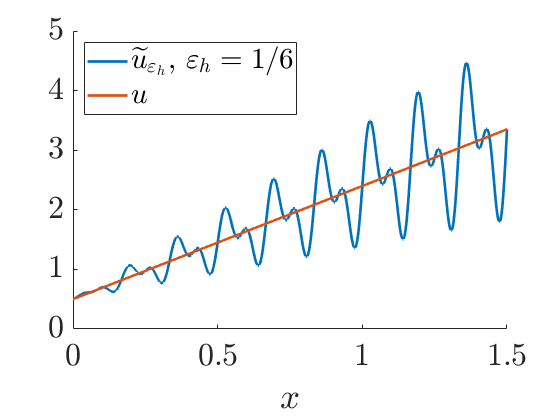}
    \includegraphics[width=0.32\linewidth]{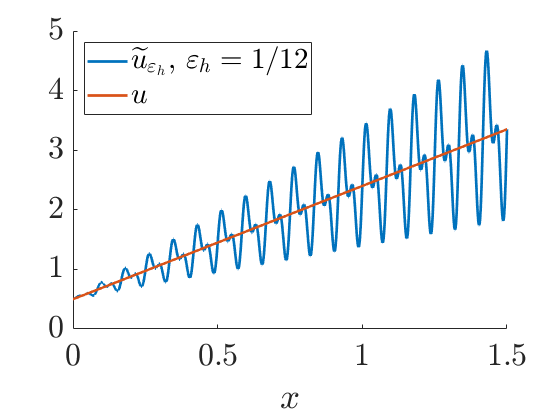}
    \includegraphics[width=0.32\linewidth]{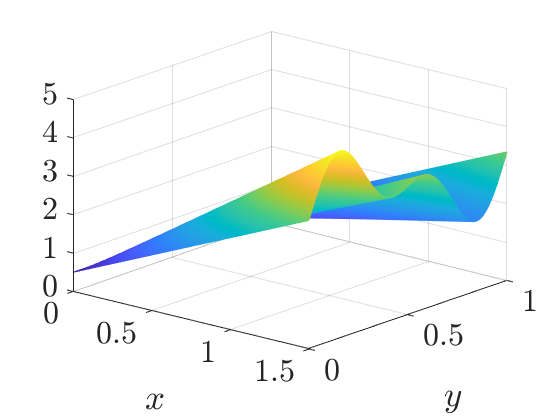}
    \caption{Elements from a sequence $\left\{ \widetilde{u}_{\varepsilon _{h}}\right\} $ for two fixed
    ${\varepsilon}_h$  together with the weak 
    $L^2(\Omega)$-limit $u$ (the two leftmost pictures) and the two-scale limit (to the right).}\label{Figur: two-scale2}
\end{figure}

The next two theorems relate the two-scale limit to the weak and strong $L^{2}(\Omega )$-limits, respectively, showing that two-scale convergence is weaker than strong convergence but stronger than weak convergence.

\begin{theorem}
\label{Th svagtvaskale}Let $\left\{ u_{\varepsilon _{h}}\right\} $ be a
sequence in $L^{2}(\Omega )$ which two-scale converges to $u_{0}\in
L^{2}(\Omega \times Y)$. Then%
\begin{equation}
u_{\varepsilon _{h}}\left( x\right) \rightharpoonup u\left( x\right)
=\int_{Y}u_{0}\left( x,y\right) dy\text{ in }L^{2}(\Omega )
\label{svag tvaskale}
\end{equation}%
and $\left\{ u_{\varepsilon _{h}}\right\} $ is bounded in $L^{2}(\Omega )$.
\end{theorem}

\begin{proof}
See, e.g., Theorem 6 in \cite{LNW}.
\end{proof}

\begin{theorem}
\label{Thtvaskalastark}
    Let $\left\{ u_{\varepsilon _{h}}\right\} $ be a sequence in $L^{2}(\Omega )$ which converges strongly to $u$ in $L^2(\Omega)$. Then
    \begin{equation*}
        u_{\varepsilon_h} \left(x\right) \overset{2}{\rightharpoonup} u\left(x\right)\text{.}
    \end{equation*}
\end{theorem}
\begin{proof}
    See, e.g., Theorem~5 in \cite{LNW}.
\end{proof}

We end this subsection by mentioning some characteristics of the test functions. We have the following theorem. 

\begin{theorem}
Let $v\in L^{2}(\Omega ;C_{\sharp }(Y))$. Then%
\begin{equation}
v\left( x,\frac{x}{\varepsilon _{h}}\right) \rightharpoonup \int_{Y}v\left( x,y\right) dy\text{ in }L^{2}(\Omega )
\label{svag testfunk}
\end{equation}%
and%
\begin{equation*}
v\left( x,\frac{x}{\varepsilon _{h}}\right) \overset{2}{\rightharpoonup }%
v\left( x,y\right) \text{.}
\end{equation*}
\end{theorem}

\begin{proof}
See, e.g., Propositions 40 and 51 in \cite{Flo}.
\end{proof}

We also give some key properties of the space of test
functions in order to obtain two-scale compactness. For all $v\in
L^{2}(\Omega ;C_{\sharp }(Y))$,

\begin{equation}
\left\Vert v\left( x,\frac{x}{\varepsilon _{h}}\right) \right\Vert
_{L^{2}(\Omega )}\leq \left\Vert v\right\Vert _{L^{2}(\Omega ;C_{\sharp
}(Y))}  \label{BoundL2(omega,Cper(Y))}
\end{equation}%
and%
\begin{equation}
\underset{h\rightarrow \infty }{\lim }\left\Vert v\left( x,\frac{x}{%
\varepsilon _{h}}\right) \right\Vert _{L^{2}(\Omega )}=\left\Vert
v\right\Vert _{L^{2}(\Omega \times Y)}\text{;}  \label{Testf L2}
\end{equation}%
see, e.g., Theorem 3 in \cite{LNW}. We will relate to the properties (\ref{svag testfunk})--(\ref{Testf L2}) in this periodic setting in a more general framework in
Sections~\ref{Operator governed} and \ref{Unification}.

\subsection{Very weak two-scale convergence}

Holmbom introduced a prototypical form of very weak two-scale convergence in 1997 in \cite{HOL}
as a tool for handling sequences not necessarily $L^{2}(\Omega )$-bounded
appearing in the homogenization of parabolic partial differential equations
exhibiting both spatial and temporal rapid oscillations. It has later been developed and formally defined in the modern sense in \cite{FHOPvery-weak},
where it also was given the name very weak multiscale convergence. We give
the two-scale version of the definition stated in \cite{FHOPvery-weak}.

\begin{definition}
A sequence $\left\{ w_{\varepsilon
_{h}}\right\} $ in $L^{1}(\Omega )$ is said to two-scale converge very
weakly to $w_{0}\in L^{1}(\Omega \times Y)$ if%
\begin{equation*}
\lim_{h\rightarrow \infty }\int_{\Omega }w_{\varepsilon _{h}}\left( x\right)
v_{1}\left( x\right) v_{2}\left( \frac{x}{\varepsilon _{h}}\right)
dx=\int_{\Omega }\int_{Y}w_{0}\left( x,y\right) v_{1}\left( x\right)
v_{2}\left( y\right) dydx
\end{equation*}%
for any $v_{1}\in D(\Omega )$ and $v_{2}\in C_{\sharp }^{\infty }(Y)/\mathbb{R}$ where%
\begin{equation}
\int_{Y}w_{0}\left( x,y\right) dy=0\text{.}
\label{villkor unikhet very weak twoscale}
\end{equation}%
This is denoted%
\begin{equation*}
w_{\varepsilon _{h}}\left( x\right) \underset{\operatorname{\mathrm{vw}}}{\overset{2}{%
\rightharpoonup }}w_{0}\left( x,y\right) \text{.}
\end{equation*}
\end{definition}

The set of test functions is a subset of those in classical two-scale convergence which together with condition (\ref{villkor unikhet very weak twoscale}) filters out the global trend (in contrast to in two-scale convergence). 

In order to see the difference between a very weak two-scale limit and a
classical two-scale limit we consider the sequence represented in the leftmost pictures in Figure~\ref{Figur: two-scale}. Figure~\ref{Figur: very
weak} shows the two-scale limit and the very weak two-scale limit for this
sequence. Note that this particular sequence is bounded and therefore
exhibits both limits while an unbounded sequence might not have a two-scale
limit but still manages to attain a very weak two-scale limit. An example of
a sequence that is not bounded in $L^{2}(\Omega )$ and does not attain a
two-scale limit is shown in Figure~\ref{Figur: very weak unbounded} together
with its very weak two-scale limit.%
\begin{figure}[H]
    \centering
    \includegraphics[width=0.4\linewidth]{twoscale1.png}
    \includegraphics[width=0.4\linewidth]{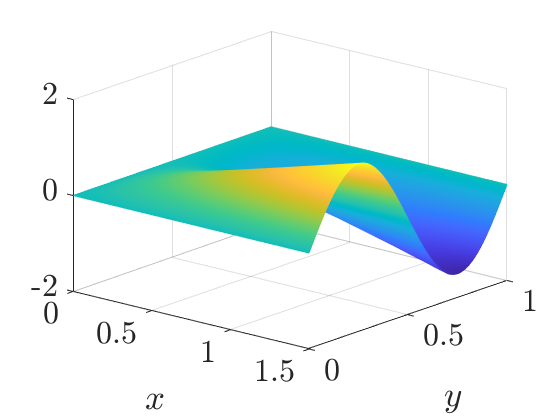}
    \caption{The two-scale limit (to the left) and the very weak
    two-scale limit (to the right) of the sequence from Figure~\ref{Figur: two-scale}.}\label{Figur: very weak}
\end{figure}
\begin{figure}[H]
    \centering
    \includegraphics[width=0.4\linewidth]{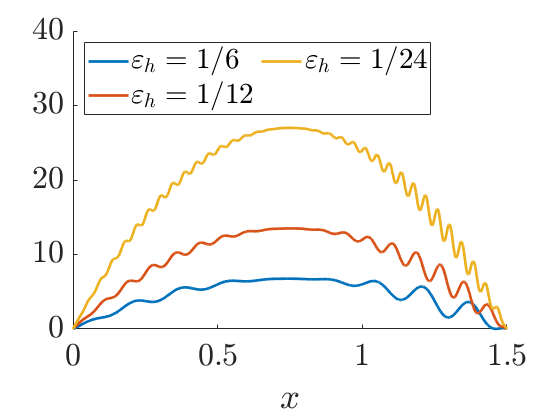}
    \includegraphics[width=0.4\linewidth]{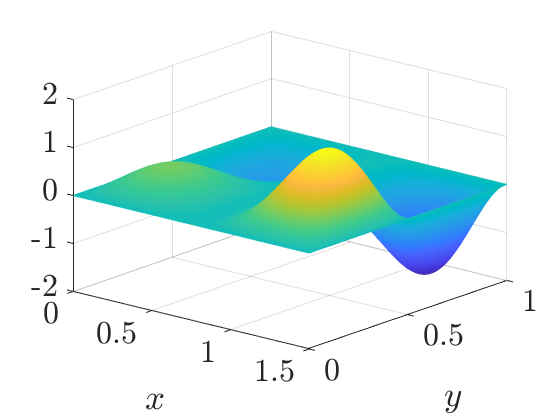}
    \caption{A sequence, not bounded in $L^2(\Omega)$, for three different
    $\varepsilon_h$ (to the left) and its very weak two-scale limit (to the
    right).}\label{Figur: very weak unbounded}
\end{figure}

We give the following compactness result which concerns sequences of the
form $\left\{ \varepsilon _{h}^{-1}u_{\varepsilon _{h}}\right\} $\ which
might not be bounded in $L^{2}(\Omega )$.

\begin{theorem}
\label{Th: compactness very weak two-scale}Let $\left\{ u_{\varepsilon
_{h}}\right\} $ be a bounded sequence in $H_{0}^{1}(\Omega )$. Then, up to a
subsequence,%
\begin{equation}
\label{eq:compvw}
\varepsilon _{h}^{-1}u_{\varepsilon _{h}}\left( x\right) \overset{2}{%
\underset{\mathrm{vw}}{\rightharpoonup }}u_{1}\left( x,y\right) \text{,}
\end{equation}%
where $u_{1}\in L^{2}(\Omega ;H_{\sharp }^{1}(Y)/\mathbb{R})$.
\end{theorem}

\begin{proof}
See Theorem~4 in \cite{FHOPvery-weak}.
\end{proof}

\begin{remark}
\label{Remark vw}
    The original approach, presented in \cite{HOL}, concerns finding a limit for the sequence $\left\{ \varepsilon _{h}^{-1}\left( u_{\varepsilon _{h}} - u\right) \right\}$ where $u$ is the weak $H_{0}^{1}(\Omega )$-limit of $\left\{ u_{\varepsilon_{h}} \right\}$. The idea behind the abbreviation of $\left\{ \varepsilon_{h}^{-1} \left( u_{\varepsilon _{h}} -u\right) \right\}$ to $\left\{ \varepsilon_{h}^{-1} u_{\varepsilon _{h}}\right\}$ can be traced to a closely related result given by Nguetseng and Woukeng in \cite{NgWo1} within the framework of $\Sigma$-convergence.
\end{remark}

\section{The periodic unfolding method\label{Unfolding}}

The periodic unfolding method was
introduced in 2002 by Cioranescu et al. in \cite{CiDaGr2} and was developed
from a method using a \textquotedblleft dilation\textquotedblright\ operator
for the homogenization of a periodic medium; see, e.g., \cite{ADH}. As for
two-scale convergence, the unfolding method also uses multiple scales but in
a different way. Here the micro-oscillations are first unfolded in a second
variable and then the passage to the limit is done using weak $L^{2}$-convergence. Similar techniques have later been used for different kinds of problems; see, e.g., \cite{BLM}, \cite{Len}, \cite{AC}, \cite{ACV} and \cite{CD}.

To establish the setting we let $\xi + Y$, $\xi \in \mathbb{Z}^N$, be a translation of the unit cube $Y$ by $\xi$ and note that $\bigcup_{\xi \in \mathbb{Z}^N} \left(\xi + \overline{Y} \right)$ covers $\mathbb{R}^N$, i.e., the set of disjoint open unit cubes $Y$ with integer vertices covers $\mathbb{R}^N$ up to a set of measure zero. For a.e. $z \in \mathbb{R}^N$ let $\left[ z\right] _{Y}$ denote the unique vector in $\mathbb{Z}^N$ such that $z - \left[ z\right]_{Y} \in Y$ and denote $\left\{ z\right\}_{Y} = z - \left[ z\right]_{Y}$. Thus,
\begin{equation}
\label{decomposition z}
    z = \left[ z\right]_{Y} + \left\{ z\right\}_{Y}\text{,}
\end{equation}
i.e., $z$ is decomposed into an integer part $\left[ z \right]_{Y}$ and a fractional part $\left\{ z \right\}_{Y}$. Scaling by $\varepsilon$ and letting $z = x/\varepsilon$ in (\ref{decomposition z}) we have
\begin{equation*}
x=\varepsilon \left( \left[ \frac{x}{\varepsilon }\right] _{Y}+\left\{ \frac{%
x}{\varepsilon }\right\} _{Y}\right) 
\end{equation*}%
where the terms form a decomposition of $x \in \mathbb{R}^N$ in terms of an $\varepsilon$-cube $\varepsilon \left( \xi + Y \right)$, $\xi \in \mathbb{Z}^N$. See Figure~\ref{Figur: decomposition z}.
\begin{figure}[H]
    \centering
    \includegraphics[width=0.49\linewidth]{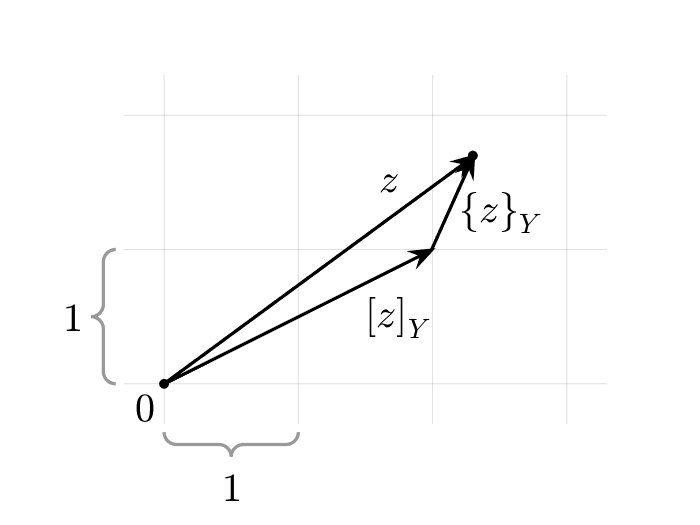}
    \includegraphics[width=0.49\linewidth]{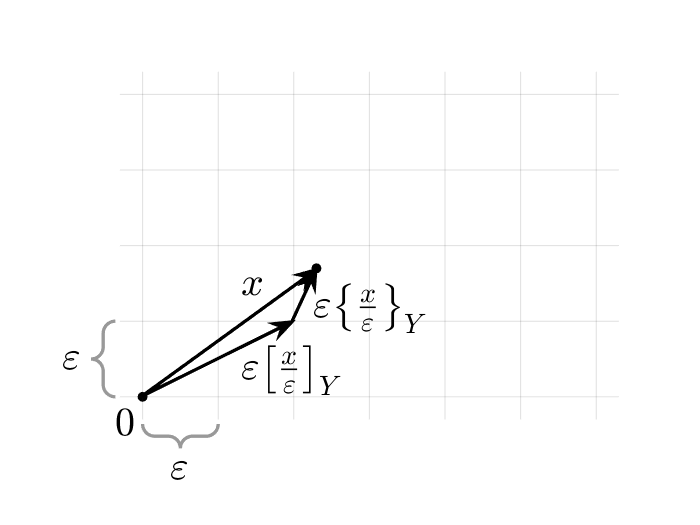}
    \caption{Decomposition of $z \in \mathbb{R}^N$ in terms of a unit cube (to the left) and decomposition of $x \in \mathbb{R}^N$ in terms of an $\varepsilon$-cube (to the right).} \label{Figur: decomposition z}  
\end{figure}
\begin{remark}
\label{rem: y-per function unfolding}
    For a $Y$-periodic function $v$ it holds that
    \begin{equation*}
        v \left( \left\{ y \right\}_Y \right)
        = v \left( y - \left[ y \right]_Y \right)
        = v \left( y \right) 
    \end{equation*}
    for a.e. $y \in Y$.
\end{remark}

Finally, we introduce the sets
\begin{equation*}
    \Xi_\varepsilon = \left\{ \xi \in \mathbb{Z}^N : \varepsilon \left( \xi + Y \right) \subset \Omega \right\} \text{,}
\end{equation*}
\begin{equation*}
    \widehat{\Omega}_\varepsilon = \mathrm{interior} \left\{ \bigcup_{\xi \in \Xi_\varepsilon} \varepsilon \left( \xi + \overline{Y} \right) \right\}
\end{equation*}
and $\Lambda_\varepsilon = \Omega \setminus \widehat{\Omega}_\varepsilon$. Informally, $\widehat{\Omega}_\varepsilon$ can be viewed as the set of whole $\varepsilon$-cubes contained in $\Omega$, together with their common boundaries, whereas $\Lambda_\varepsilon$ consists of the parts of the $\varepsilon$-cubes intersecting $\partial \Omega$ that lie within $\Omega$, see Figure~\ref{Figur: partition of omega}.
\begin{figure}[H]
    \centering
    \includegraphics[width=0.43\linewidth]{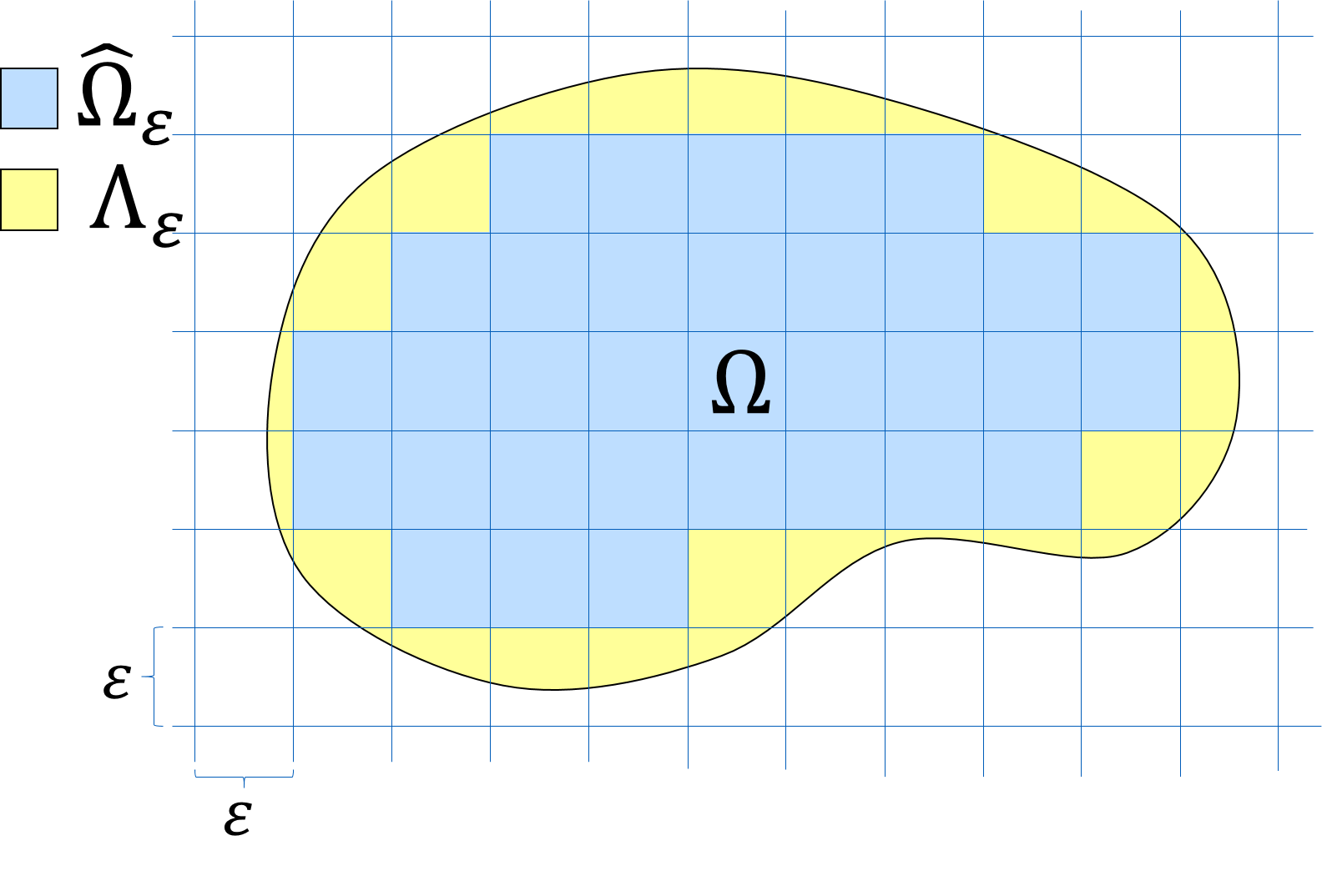}
    \caption{Partition of $\Omega$.} \label{Figur: partition of omega}  
\end{figure}
\begin{remark}
    The set $\Lambda_\varepsilon$ satisfies
    \begin{equation*}
        \Lambda_\varepsilon \subset \left\{ x \in \Omega : \mathrm{dist} \left( x, \partial \Omega \right) < \varepsilon \sqrt{N} \right\}
    \end{equation*}
    and, since $\partial \Omega$ is bounded,
    \begin{equation*}
        \lim_{\varepsilon \to 0} \left| \Lambda_\varepsilon \right| = 0
    \end{equation*}
    where $\left| \Lambda_\varepsilon \right|$ is the Lebesgue measure of $\Lambda_\varepsilon$. See, e.g., Remark~1.1 in \cite{CiDaGr3}.
\end{remark}
Following \cite{CiDaGr1} we define the unfolding operator $%
\mathcal{T}_{\varepsilon }$ as follows.
\begin{definition}
For $u$ Lebesgue-measurable in $\Omega $, the unfolding operator $\mathcal{T}%
_{\varepsilon }$ is defined as%
\begin{equation*}
\mathcal{T}_{\varepsilon} \left( u \right) \left( x,y \right) = \left \{
\begin{aligned}
&u \left( \varepsilon \left[ \frac{x}{\varepsilon} \right]_Y + \varepsilon y \right) && \text{for a.e. $\left( x,y \right) \in \widehat{\Omega}_{\varepsilon} \times Y$,} \\
&0 && \text{for a.e. $\left( x,y \right) \in \Lambda_{\varepsilon} \times Y$.}
\end{aligned} \right.
\end{equation*}
\end{definition}

Note that $\mathcal{T}_{\varepsilon }\left( u\right) $ is
Lebesgue-measurable in $\Omega \times Y$. Note also that as long as $x$
remains in the same $\varepsilon $-cell, i.e., on each $Y_{i}^{\varepsilon
}\times Y$, $\mathcal{T}_{\varepsilon }\left( u\right) $ is constant with
respect to $x$, see Figure~\ref{Figur: unfolding} which shows the
unfoldings of the oscillating functions presented in Figure~\ref{Figur: two-scale}.%
\begin{figure}[H]
    \centering
    \includegraphics[width=0.4\linewidth]{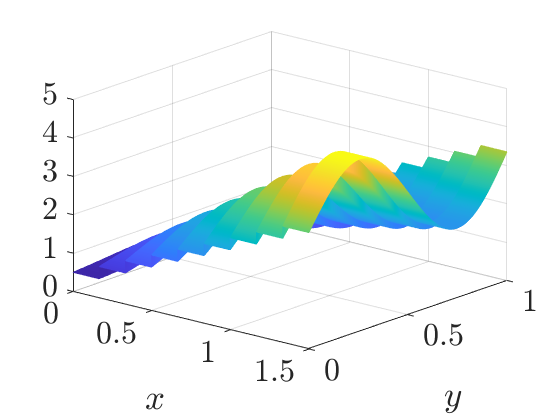}
    \includegraphics[width=0.4\linewidth]{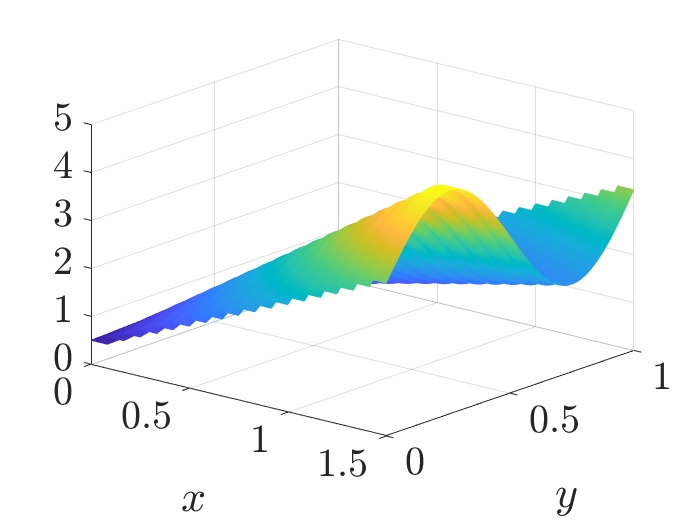}
    \caption{The unfoldings for the functions $u_{\varepsilon _{h}}$ in the two leftmost pictures in Figure~\ref{Figur:
    two-scale}.}\label{Figur: unfolding}
\end{figure}

The following proposition is essential for obtaining compactness of the unfolded sequence.

\begin{proposition}
\label{Prop5}The operator $\mathcal{T}_{\varepsilon }:L^{2}(\Omega
)\rightarrow L^{2}(\Omega \times Y)$ is linear and continuous. For every $%
u\in L^{2}(\Omega )$,%
\begin{equation}
\label{unfolding_linj_kont}
\left\Vert \mathcal{T}_{\varepsilon }\left( u\right) \right\Vert
_{L^{2}(\Omega \times Y)}\leq \left\Vert u\right\Vert _{L^{2}(\Omega )}\text{.}
\end{equation}
\end{proposition}

\begin{proof}
See Proposition~1.8 (iii) in \cite{CiDaGr3}.
\end{proof}

\begin{remark}
    For $u$ restricted to $\widehat{\Omega }_{\varepsilon}$ we have equality in (\ref{unfolding_linj_kont}). See the proof of Proposition~1.8~(iii) in \cite{CiDaGr3}.
\end{remark}

If $\left\{ u_{\varepsilon _{h}}\right\}$ is a bounded sequence in $L^{2}(\Omega )$, by Proposition~\ref{Prop5}, $\left\{ \mathcal{T}_{\varepsilon _{h}}\left( u_{\varepsilon _{h}}\right) \right\}$ is bounded in the space $L^{2}(\Omega \times Y)$. By weak compactness in $L^{2}(\Omega \times Y)$, we get, up to a subsequence,
\begin{equation}
\mathcal{T}_{\varepsilon _{h}}\left( u_{\varepsilon _{h}}\right) \left(
x,y\right) \rightharpoonup u_{0}\left( x,y\right) \text{ in }L^{2}(\Omega
\times Y).
\label{unfoldcompactformel}
\end{equation}%

Since $u_{0}$\ here is the same as the two-scale limit in Theorem~\ref{Th
twoskalecompact} the weak $L^{2}(\Omega \times Y)$-limit of the
unfolded sequence illustrated in Figure~\ref{Figur: unfolding} is
identical to the function depicted to the right in Figure~\ref{Figur: two-scale}.
See further Section~\ref{Unification} and in particular Proposition~\ref{Pr
ekvivalens}.

In line with Theorem~\ref{Th svagtvaskale}, we have the following relation between the weak $L^{2}(\Omega \times Y)$-limit of the unfolded sequence and
the weak $L^{2}(\Omega )$-limit of the sequence.

\begin{proposition} \label{Prop: unfolding gives weak}
Let $\left\{ u_{\varepsilon _{h}}\right\} $
be a sequence in $L^{2}(\Omega )$ such that $\left\{ \mathcal{T}%
_{\varepsilon _{h}}\left( u_{\varepsilon _{h}}\right) \right\} $ converges
weakly to $u_{0}$ in $L^{2}(\Omega \times Y)$. Then%
\begin{equation*}
u_{\varepsilon _{h}}\left( x\right) \rightharpoonup
\int\nolimits_{Y}u_{0}\left( x,y\right) dy\text{ in }L^{2}(\Omega )\text{.}
\end{equation*}
\end{proposition}

\begin{proof}
See Proposition 2.9 (iii) in \cite{CiDaGr1}.
\end{proof}

Analogous to Theorem \ref{Thtvaskalastark}, we also have the following relation to the strong $L^{2}(\Omega )$-limit.

\begin{theorem}
Let $\left\{ u_{\varepsilon _{h}}\right\} $ be a
sequence in $L^{2}(\Omega )$ which converges strongly to $u\in L^{2}(\Omega
) $. Then%
\begin{equation*}
\mathcal{T}_{\varepsilon _{h}}\left( u_{\varepsilon _{h}}\right) \left(
x,y\right) \rightarrow u\left( x\right) \text{ in }L^{2}(\Omega \times Y).
\end{equation*}
\end{theorem}

\begin{proof}
See Proposition 2.9 (ii) in \cite{CiDaGr1}.
\end{proof}

\begin{remark}
\label{Rem: Unfoldalt}
    An alternative way of defining the unfolding operator is
    \begin{equation*}
        \mathcal{T}_{\varepsilon} \left( u \right) \left( x,y \right) = \left \{
        \begin{aligned}
            &u \left( \varepsilon \left[ \frac{x}{\varepsilon} \right]_Y + \varepsilon y \right) && \text{for a.e. $\left( x,y \right) \in \widehat{\Omega}_{\varepsilon} \times Y$,} \\
            &u\left( x \right) && \text{for a.e. $\left( x,y \right) \in \Lambda_{\varepsilon} \times Y$,}
        \end{aligned} \right.
    \end{equation*}
    which implies that the integral is preserved, i.e.,
    \begin{equation*}
        \int_\Omega\int_Y \mathcal{T}_\varepsilon \left( u \right) \left( x,y \right) \mathrm{d}y \mathrm{d}x
        = \int_\Omega u \left( x \right) \mathrm{d}x
    \end{equation*}
    while the corresponding equality does not hold for the original definition.
    In the case where $\left| \Lambda_\varepsilon\right| = 0$ or where $\left| \Lambda_\varepsilon\right| > 0$, $\left| \Lambda_{\varepsilon_n}\right| \to 0$ as $n \to \infty$ and $\left\{ u_{\varepsilon_n} \right\}$ is bounded in $L^2(\Omega)$, the two definitions are equivalent. For a detailed investigation of this approach, see \cite{Francu2010}, \cite{Francu2011} by Franc\r{u} and \cite{FrSv2012} by Franc\r{u} and Svanstedt. Other works where this approach has been used are, e.g., \cite{Nes} and \cite{BoBr}.
\end{remark}

We introduce the so-called averaging operator which we return to in Section~\ref{Unification} where, together with the concept of general two-scale convergence (introduced in Section~\ref{Operator governed}), it will provide us with a unifying bridge between two-scale convergence and unfolding. We have the following definition, given in \cite{CiDaGr1}.

\begin{definition}
The averaging operator $\mathcal{U}_{\varepsilon }:L^{2}(\Omega \times
Y)\rightarrow L^{2}(\Omega )$ is defined as%
\begin{equation*}
\mathcal{U}_{\varepsilon} \left( v \right) \left( x \right) = \left \{
\begin{aligned}
&\int_Y v \left( \varepsilon \left[ \frac{x}{\varepsilon} \right]_Y + \varepsilon z, \left\{ \frac{x}{\varepsilon} \right\}_Y \right) dz && \text{for a.e. $x \in \widehat{\Omega}_{\varepsilon}$,} \\
&0 && \text{for a.e. $x \in \Lambda_{\varepsilon}$.}
\end{aligned} \right.
\end{equation*}
\end{definition}

Figure~\ref{Figur: averaging operator} illustrates $\mathcal{U}_\varepsilon$ operating on a function $v$ in $L^2(\Omega \times Y)$.
\begin{figure}[H]
    \centering
    \includegraphics[width=0.32\linewidth]{twoscale1.png}
    \includegraphics[width=0.32\linewidth]{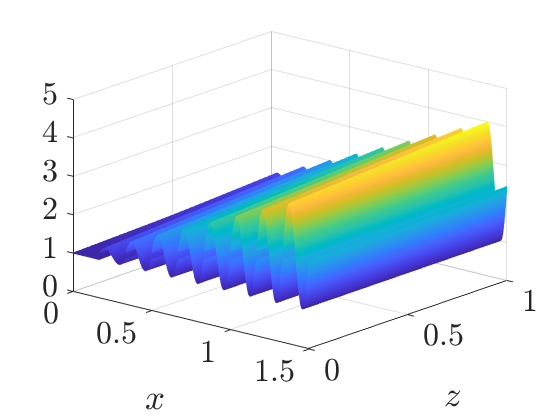}
    \includegraphics[width=0.32\linewidth]{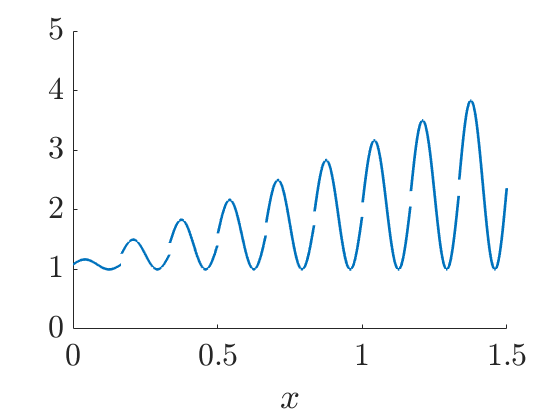}
    \caption{A function $v \in L^2(\Omega \times Y)$ (to the left), $v \left( \varepsilon \left[ \frac{x}{\varepsilon}\right]_Y + \varepsilon z, \left\{ \frac{x}{\varepsilon} \right\}_Y  \right)$ (in the middle) and $\mathcal{U}_{\varepsilon} \left( v \right)$ (to the right) for $\varepsilon = 1/6$.} \label{Figur: averaging operator}
\end{figure}

The following proposition provides us with a formula connecting the
unfolding operator and the averaging operator.

\begin{proposition}
\label{Prop: adjoint}The averaging operator $\mathcal{U}_{\varepsilon }$ is
the adjoint operator of the unfolding operator $\mathcal{T}_{\varepsilon }$,
i.e.,%
\begin{equation*}
\int_{\Omega }\int_{Y}\mathcal{T}_{\varepsilon }\left( u\right) \left(
x,y\right) v\left( x,y\right) dydx=\int_{\Omega }u\left( x\right) \mathcal{U}%
_{\varepsilon }\left( v\right) \left( x\right) dx  
\end{equation*}%
for any $u\in L^{2}(\Omega )$ and $v\in L^{2}(\Omega \times Y)$.
\end{proposition}

\begin{proof}
See Subsection 2.2 in \cite{CiDaGr1}.
\end{proof}

The next result plays a similar role for unfolding as (\ref%
{BoundL2(omega,Cper(Y))}) and (\ref{Testf L2}) do for two-scale convergence.

\begin{proposition}
\label{Prop217}$\phantom{}$
The following assertions hold.
\begin{enumerate}
\item[(i)] The averaging operator $\mathcal{U}_{\varepsilon }$ is linear and
continuous from $L^{2}(\Omega \times Y)$ to $L^{2}(\Omega )$ and, for every
$v\in L^{2}(\Omega \times Y)$,%
\begin{equation*}
\left\Vert \mathcal{U}_{\varepsilon }\left( v\right) \right\Vert
_{L^{2}(\Omega )}\leq \left\Vert v\right\Vert _{L^{2}(\Omega \times Y)}\text{%
.}
\end{equation*}

\item[(ii)] For every $v\in L^{2}(\Omega \times Y)$,%
\begin{equation*}
\mathcal{T}_{\varepsilon_h }\left( \mathcal{U}_{\varepsilon_h }\left( v\right)
\right) \left( x,y\right) \rightarrow v\left( x,y\right) \text{ in }%
L^{2}(\Omega \times Y)\text{.}
\end{equation*}

\item[(iii)] For any $v\in L^{2}(\Omega \times Y)$ it holds that%
\begin{equation*}
\lim_{h\rightarrow \infty }\left\Vert \mathcal{U}_{\varepsilon _{h}}\left(
v\right) \right\Vert _{L^{2}(\Omega )}=\left\Vert v\right\Vert
_{L^{2}(\Omega \times Y)}\text{.}
\end{equation*}
\end{enumerate}
\end{proposition}

\begin{proof}

(i) See Proposition 2.17 in \cite{CiDaGr1}.

(ii) See Proposition 1.29 (i) in \cite{CiDaGr3}.

(iii) From Proposition~\ref{Prop: adjoint} we have%
\begin{align*}
    \lim_{h\rightarrow \infty }\left\Vert \mathcal{U}_{\varepsilon _{h}}\left(
    v\right) \right\Vert _{L^{2}(\Omega )}^{2}
    &= \lim_{h\rightarrow \infty
    }\int_{\Omega }\left( \mathcal{U}_{\varepsilon _{h}}\left( v\right)(x) \right)
    ^{2}dx \\
    &= \lim_{h\rightarrow \infty }\int_{\Omega }\int_{Y}\mathcal{T}_{\varepsilon
    _{h}}\left( \mathcal{U}_{\varepsilon _{h}}\left( v\right) \right) \left(
    x,y\right) v\left( x,y\right) dydx\text{.}
\end{align*}%
Assertion $\left( ii\right) $ implies that $\mathcal{T}_{\varepsilon
_{h}}\left( \mathcal{U}_{\varepsilon _{h}}\left( v\right) \right) $
converges weakly in $L^{2}(\Omega \times Y)$ which in turn implies%
\begin{gather*}
\lim_{h\rightarrow \infty }\int_{\Omega }\int_{Y}\mathcal{T}_{\varepsilon
_{h}} \! \left( \mathcal{U}_{\varepsilon _{h}} \! \left( v\right) \right) \! \left(
x,y\right) v \! \left( x,y\right) dydx 
=\int_{\Omega }\int_{Y}\left( v \! \left( x,y\right) \right) ^{2}dydx=\left\Vert
v\right\Vert _{L^{2}(\Omega \times Y)}^{2}
\end{gather*}%
and the statement follows.
\end{proof}

We have the following relation between the weak $L^{2}(\Omega )$-limit of
the averaging operator and the weak $L^{2}(\Omega \times Y)$-limit of the
sequence.

\begin{proposition}
\label{Prop218}Let $\left\{ u_{\varepsilon _{h}}\right\} $ be a sequence
which converges weakly to $u$ in $L^{2}(\Omega \times Y)$. Then%
\begin{equation*}
\mathcal{U}_{\varepsilon _{h}}\left( u_{\varepsilon _{h}}\right) \left(
x\right) \rightharpoonup \int\nolimits_{Y}u\left( x,y\right) dy\text{ in }%
L^{2}(\Omega )\text{.}
\end{equation*}
\end{proposition}

\begin{proof}
See Proposition 2.18 (i) in \cite{CiDaGr1}.
\end{proof}

We end this section by giving some results which we discuss further in
Section~\ref{Unification}. First,\ we introduce the local average operator according to Definition~1.21 in \cite{CiDaGr3}.
\begin{definition}
    The local average operator $\mathcal{M}_\varepsilon : L^2(\Omega) \to L^2(\Omega)$ is defined as
    \begin{equation*}
    \mathcal{M}_{\varepsilon} \left( u \right) \left( x \right) = \left\{
    \begin{aligned}
        &\int_Y u \left( \varepsilon \left[ \frac{x}{\varepsilon} \right]_Y + \varepsilon z \right) dz 
        && \text{for $x \in \widehat{\Omega}_{\varepsilon}$,} \\
        &0 
        && \text{for $x \in \Lambda_{\varepsilon}$.}
    \end{aligned} \right.
\end{equation*}
\end{definition}

The local average operator $\mathcal{M}_{\varepsilon }$ is illustrated in Figure~\ref{Figur: local average operator}.
\begin{figure}[H]
    \centering
    \includegraphics[width=0.4\linewidth]{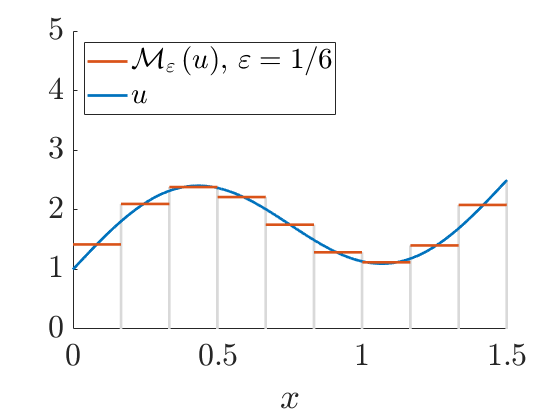}
    \includegraphics[width=0.4\linewidth]{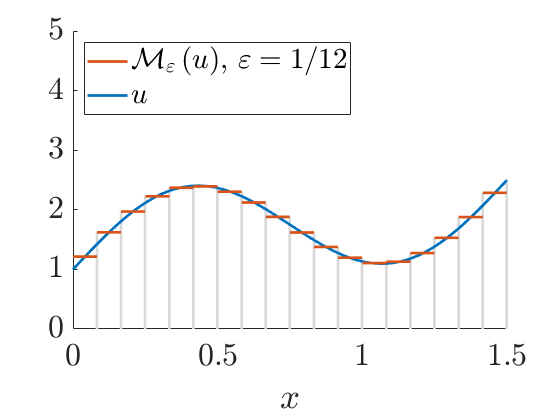}
    \caption{A function $u$ and $\mathcal{M}_{\varepsilon} \left( u \right)$ for two different $\varepsilon$ values.}\label{Figur: local average operator}
\end{figure}
\begin{remark}
    The local average operator is related to the unfolding operator by
    \begin{equation*}
        \mathcal{M}_{\varepsilon }\left( u\right) \left( x\right) =\int\nolimits_{Y} \mathcal{T}_{\varepsilon }\left( u\right) \left( x,y\right) dy \textit{.}
    \end{equation*}
    See, e.g., Remark~2.23 in \cite{CiDaGr1}.
\end{remark}
\begin{remark}
\label{rem: adjoint lika med local average}
    For $v \in L^2(\Omega) \subset L^2(\Omega \times Y)$, i.e., $v$ depending only on $x$, not $y$,
    \begin{equation*}
        \mathcal{U}_\varepsilon \left( v \right) \left( x \right) = \mathcal{M}_\varepsilon \left( u \right) \left( x \right) \text{.}
    \end{equation*}
    See, e.g., Remark~2.24 in \cite{CiDaGr1}.
\end{remark}

We have the following result.
\begin{theorem}
\label{Thunfoldampli}Let $\left\{ u_{\varepsilon _{h}}\right\} $ be a
bounded sequence in $H^{1}(\Omega )$. Then, up to a subsequence,%
\begin{equation}
\label{vw unfolding}
\varepsilon _{h}^{-1}\left( \mathcal{T}_{\varepsilon _{h}}\left(
u_{\varepsilon _{h}}\right) \left( x,y\right) -\mathcal{M}_{\varepsilon
_{h}}\left( u_{\varepsilon _{h}}\right) \left( x\right) \right)
\rightharpoonup u_{1}\left( x,y\right) +y^{c}\cdot \nabla u\left( x\right)
\end{equation}%
in $L^{2}(\Omega ;H^{1}(Y))$, where $u\in H^{1}(\Omega )$, $u_{1}\in
L^{2}(\Omega ;H_{\sharp }^{1}(Y)/\mathbb{R})$ and \newline
$y^{c}=\left( y_{1}-1/2,y_{2}-1/2,...,y_{N}-1/2\right) $.
\end{theorem}

\begin{proof}
See Theorem 3.5 in \cite{CiDaGr1}.
\end{proof}

For strongly convergent sequences we have the following result.

\begin{proposition}
\label{Prop225}Let $\left\{ u_{\varepsilon _{h}}\right\} $ be a sequence which converges strongly to $u$ in $L^{2}(\Omega )$. Then%
\begin{equation*}
\mathcal{M}_{\varepsilon _{h}}\left( u_{\varepsilon _{h}}\right) \left(
x\right) \rightarrow u\left( x\right) \text{ in }L^{2}(\Omega )\text{.}
\end{equation*}
\end{proposition}

\begin{proof}
See Proposition 2.25 (iii) in \cite{CiDaGr1}.
\end{proof}

\begin{remark}
\label{Rem: M0 ger ueps0}If $\left\{u_{\varepsilon_h}\right\}$ converges strongly to $u$ in $L^2(\Omega)$ and  $\mathcal{M}_{\varepsilon _{h}}\left(
u_{\varepsilon _{h}}\right) =0$, Proposition~\ref{Prop225} implies that $u=0$.
\end{remark}

Finally, we give some convergence results for the unfolded sequence in\linebreak
$L^{2}(\Omega ;H^{1}(Y))$.

\begin{theorem}
Let $\left\{ u_{\varepsilon
_{h}}\right\} $ be a sequence which converges weakly to $u$ in $H^{1}(\Omega
)$. Then%
\begin{equation*}
\mathcal{T}_{\varepsilon _{h}}\left( u_{\varepsilon _{h}}\right) \left(
x,y\right) \rightharpoonup u\left( x\right) \text{ in }L^{2}(\Omega
;H^{1}(Y))\text{.}
\end{equation*}%
Furthermore, if $\left\{ u_{\varepsilon _{h}}\right\} $ converges strongly
in $L^{2}(\Omega )$ then%
\begin{equation*}
\mathcal{T}_{\varepsilon _{h}}\left( u_{\varepsilon _{h}}\right) \left(
x,y\right) \rightarrow u\left( x\right) \text{ in }L^{2}(\Omega ;H^{1}(Y))%
\text{.}
\end{equation*}
\end{theorem}

\begin{proof}
See Corollary 3.3 in \cite{CiDaGr1}.
\end{proof}

\section{General two-scale convergence\label{Operator governed}}

This section is devoted to general two-scale convergence where the limit process depends on so-called two-scale operators. This type of generalization of two-scale convergence was introduced by Holmbom in 1996 in \cite{HOL1}; see also \cite{HSSW}, \cite{HOS} and \cite{Sil}. We give the following definitions.

\begin{definition}
A linear operator%
\begin{equation*}
\tau :X\rightarrow L^{2}(\Omega )\text{,}
\end{equation*}%
where $X\subset L^{2}(\Omega \times Y)$, is called a two-scale operator.
\end{definition}

\begin{definition}
A sequence $\left\{ u_{h}\right\} $ in $%
L^{2}(\Omega )$ is said to two-scale converge to a limit $u_{0}\in
L^{2}(\Omega \times Y)$ with respect to a sequence\ $\left\{ \tau
_{h}\right\} $ of two-scale operators if%
\begin{equation*}
\lim_{h\rightarrow \infty }\int_{\Omega }u_{h}\left( x\right) \tau
_{h}\left( v\right) \left( x\right) dx=\int_{\Omega }\int_{Y}u_{0}\left(
x,y\right) v\left( x,y\right) dydx
\end{equation*}%
for all $v\in X\subset L^{2}(\Omega \times Y)$.
\end{definition}

\begin{remark}
In general the auxiliary set is not restricted to be the unit cube $Y$.\ In
fact, it does not even have to be of the same dimension as $\Omega $. It
suffices that it is an open bounded set of any dimension; see any of the
references above.
\end{remark}

General two-scale convergence may be applicable to sequences where two-scale
convergence, very weak two-scale convergence and unfolding are not. An
example where two-scale convergence, very weak two-scale convergence and
unfolding are not immediately applicable are provided in Figure~\ref{Figur:
general twoscale} which shows elements of a sequence and its general
two-scale limit.%
\begin{figure}[H]
    \centering
    \includegraphics[width=0.32\linewidth]{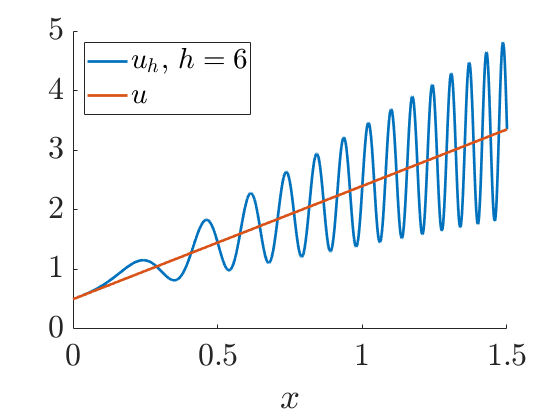}
    \includegraphics[width=0.32\linewidth]{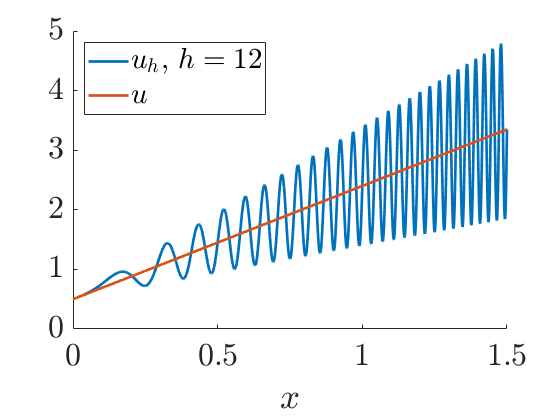}
    \includegraphics[width=0.32\linewidth]{twoscale1.png}
    \caption{Elements $u_h$ for two fixed values of $h$ from a non-periodic
    sequence together with its weak $L^2(\Omega)$-limit and, to the right, its general
    two-scale limit.}\label{Figur: general twoscale}
\end{figure}

We need to specify the properties of $\left\{ \tau _{h}\right\} $ and $X$
that make it possible to attain compactness of a similar kind as is
obtained for classical periodic two-scale convergence. Following \cite{Sil}\
we give the following definition.

\begin{definition}
\label{two-scale compatible}A sequence of operators $\left\{ \tau
_{h}\right\} $ is said to be two-scale compatible with respect to the normed
space $X\subset L^{2}(\Omega \times Y)$ if%
\begin{equation*}
\tau _{h}:X\rightarrow L^{2}(\Omega )
\end{equation*}%
are linear and there exist $C,D>0$ independent of $h$ such that%
\begin{equation}
\left\Vert \tau _{h}\left( v\right) \right\Vert _{L^{2}(\Omega )}\leq
C\left\Vert v\right\Vert _{X}  \label{Bound X}
\end{equation}%
and%
\begin{equation}
\lim_{h\rightarrow \infty }\left\Vert \tau _{h}\left( v\right) \right\Vert
_{L^{2}(\Omega )}\leq D\left\Vert v\right\Vert _{L^{2}(\Omega \times Y)}
\label{BoundL2gen}
\end{equation}%
for all $v\in X$.
\end{definition}

We state the following compactness theorem.

\begin{theorem}
\label{th14}Let $\left\{ u_{h}\right\} $ be a bounded sequence in $%
L^{2}(\Omega )$ and $\left\{ \tau _{h}\right\} $ be two-scale compatible
with respect to the separable Banach space $X\subset L^{2}(\Omega \times Y)$%
. Then there exists a subsequence of $\left\{ u_{h}\right\} $ which
two-scale converges with respect to $\left\{ \tau _{h}\right\} $ to some $%
u_{0}\in L^{2}(\Omega \times Y)$.
\end{theorem}

\begin{proof}
See Theorem 3 in \cite{HSSW}.
\end{proof}

\begin{remark}
\label{rem: general twoscale limit unique}
If $X$ in Theorem~\ref{th14} is dense in $L^{2}(\Omega \times Y)$, the limit 
$u_{0}$ is unique; see Remark~2 in Section~2.3 in \cite{HOL1}.
\end{remark}

We note that (\ref{Bound X}) and (\ref{BoundL2gen}) carry the same essential
properties as (\ref{BoundL2(omega,Cper(Y))}) and (\ref{Testf L2}),
respectively. Hence, as Theorem~\ref{th14} showed, if $X$ is a separable Banach space, we may pass to the
limit in the same way as for the periodic case in Section~\ref{Two-scale}.
However, the properties in Definition~\ref{two-scale compatible}
are insufficient to obtain connections between the conventional weak limits
and the two-scale type limits corresponding to (\ref{svag tvaskale}) and (%
\ref{svag testfunk}) why we give the following definition; see Definition~10
in \cite{Sil}.

\begin{definition}
\label{stark}Let $\left\{ \tau _{h}\right\} $ be two-scale compatible with
respect to $X\subset L^{2}(\Omega \times Y)$ and assume that $\left\{
u_{h}\right\} $ is a bounded sequence in $L^{2}(\Omega )$ two-scale
converging to $u_{0}$ with respect to $\left\{ \tau _{h}\right\} $. Then we
say that $\left\{ \tau _{h}\right\} $ is strongly two-scale compatible with
respect to $X$ if%
\begin{equation}
u_{h}\left( x\right) \rightharpoonup \int_{Y}u_{0}\left( x,y\right) dy\text{
in }L^{2}(\Omega )  \label{Strongcomp1}
\end{equation}%
and%
\begin{equation}
\tau _{h}\left( v\right) \left( x\right) \rightharpoonup \int_{Y}v\left(
x,y\right) dy\text{ in }L^{2}(\Omega )  \label{Strongcomp2}
\end{equation}%
for all $v\in X$.
\end{definition}

We return to Definition~\ref{stark} in the next section.

\section{Junctions between two-scale convergence and unfolding\label%
{Unification}}

In this section we investigate the relation
between two-scale convergence and the periodic unfolding method. The fact that two-scale convergence and unfolding are
equivalent for sequences bounded in $L^{2}(\Omega )$, as stated in the
proposition below, was established already in Proposition~2 in \cite{CiDaGr2}
and proven in Proposition~2.14 in \cite{CiDaGr1}.

\begin{proposition}
\label{Pr ekvivalens}Let $\left\{ u_{\varepsilon _{h}}\right\} $\ be a
bounded sequence in $L^{2}(\Omega )$. Then the following assertions are
equivalent:

\begin{enumerate}
\item[(i)] \qquad $u_{\varepsilon _{h}}\left( x\right) \overset{2}{\rightharpoonup }u_{0}\left(
x,y\right) \text{,}$

\item[(ii)] \qquad $\mathcal{T}_{\varepsilon _{h}}\left( u_{\varepsilon _{h}}\right) \left(
x\right) \rightharpoonup u_{0}\left( x,y\right) \text{ in }L^{2}(\Omega
\times Y)\text{.}$
\end{enumerate}
\end{proposition}

What we will see here is that both can be seen as special cases of general
two-scale convergence, but with a more restrictive class of admissible test
functions for periodic two-scale convergence. We also discuss possible similarities between unfolding and very weak two-scale convergence.

\begin{remark}
Aspects of the relation between two-scale convergence and unfolding as well as general two-scale convergence are also addressed by Franc\r{u} and Svanstedt in \cite{FrSv2012}, based on the modified definition of unfolding as stated in Remark~\ref{Rem:  Unfoldalt}.
\end{remark}

\subsection{A link through general two-scale convergence}

Classical two-scale convergence is obtained as a special case of general two-scale convergence if we choose the sequence of two-scale operators $%
\left\{ \tau _{h}\right\} $ as 
\begin{equation*}
\tau _{h}\left( v\right) \left( x\right) =v\left( x,\frac{x}{\varepsilon _{h}%
}\right), \quad x \in \Omega, 
\end{equation*}%
with $v\in X=L^{2}(\Omega ;C_{\sharp }(Y))$; see, e.g., Section~2.3 in \cite%
{HOL1}. The theorem below shows that unfolding also is a special case of general two-scale convergence if we choose%
\begin{equation*}
\tau _{h}\left( v\right) \left( x\right) =\mathcal{U}_{\varepsilon
_{h}}\left( v\right) \left( x\right), \quad x \in \Omega,
\end{equation*}%
with $v\in X=L^{2}(\Omega \times Y)$.

\begin{theorem}
\label{Pr adj tvaskalekomp}Let $\left\{ u_{\varepsilon _{h}}\right\} $ be a
bounded sequence in $L^{2}(\Omega )$. Then, for $X=L^{2}(\Omega \times Y)$,
there exists a subsequence of $\left\{ u_{\varepsilon _{h}}\right\} $ which
two-scale converges with respect to $\left\{ \mathcal{U}_{\varepsilon
_{h}}\right\} $ to some $u_{0}\in L^{2}(\Omega \times Y)$. Furthermore, $%
\left\{ \mathcal{U}_{\varepsilon _{h}}\right\} $ is strongly two-scale
compatible with respect to $X$\ and the limit $u_{0}$\ is the same as for
unfolding and periodic two-scale convergence.
\end{theorem}
\begin{proof}
By Proposition~\ref{Prop217}~(i) and~(iii), $\left\{ \mathcal{U}_{\varepsilon_{h}} \right\}$ is two-scale compatible with respect to $L^2(\Omega \times Y)$. Thus, Theorem~\ref{th14} and Remark~\ref{rem: general twoscale limit unique} yields that, for $X=L^2(\Omega \times Y)$, there exists a subsequence of $\left\{ u_{\varepsilon_h} \right\}$ which two-scale converges with respect to $\left\{ \mathcal{U}_{\varepsilon_h} \right\}$ to some unique $u_0 \in L^2( \Omega \times Y)$, i.e.,
\begin{equation}
\lim_{h\rightarrow \infty }\int_{\Omega }u_{\varepsilon _{h}}\left( x\right) 
\mathcal{U}_{\varepsilon _{h}}\left( v\right) \left( x\right)
dx=\int_{\Omega }\int_{Y}u_{0}\left( x,y\right) v\left( x,y\right) dydx\text{%
,}  \label{shestj}
\end{equation}%
up to a subsequence, for all $v\in X=L^{2}(\Omega \times Y)$.

We continue by showing that $\left\{ \mathcal{U}_{\varepsilon _{h}}\right\} $ is
strongly two-scale compatible with respect to $X=L^{2}(\Omega \times Y)$. By Proposition~\ref{Prop: adjoint} and (\ref{shestj}), for the introduced subsequence,
\begin{align*}
\lim_{h \to \infty} \int_{\Omega }\int_{Y}\mathcal{T}_{\varepsilon _{h}}\left( u_{\varepsilon
_{h}}\right) \left( x,y\right) v\left( x,y\right) dydx
&= \lim_{h \to \infty} \int_{\Omega
}u_{\varepsilon _{h}}\left( x\right) \mathcal{U}_{\varepsilon _{h}}\left(
v\right) \left( x\right) dx \\
&= \int_{\Omega} \int_Y u_0 \left( x,y \right) v \left( x,y \right) dydx
\end{align*}
for all $v \in L^{2}(\Omega \times Y)$. We also have, by (\ref{unfoldcompactformel}), that
\begin{equation*}
\lim_{h\rightarrow \infty }\int_{\Omega }\int_{Y}\mathcal{T}_{\varepsilon
_{h}}\left( u_{\varepsilon _{h}}\right) \left( x,y\right) v\left( x,y\right)
dydx=\int_{\Omega }\int_{Y}\widetilde{u}_{0}\left( x,y\right) v\left(
x,y\right) dydx\text{,}
\end{equation*}
up to a subsequence of the introduced subsequence, for some $\widetilde{u}_{0}\in L^{2}(\Omega \times Y)$ and all $v\in L^{2}(\Omega \times Y)$. By uniqueness of weak limits we have that $\widetilde{u}_{0} = u_{0}$, and thus Proposition~\ref{Prop: unfolding gives weak} yields (\ref{Strongcomp1}). Condition (\ref{Strongcomp2}) follows directly from Proposition~\ref{Prop218} and hence we have shown that $\left\{ \mathcal{U}_{\varepsilon_{h}} \right\}$ is strongly two-scale compatible.

Finally, by Proposition~\ref{Pr ekvivalens} the limit $u_{0}$ is the same as for periodic two-scale convergence.
\end{proof}

Summing up, we have two ways to obtain the classical two-scale limit from the notion of general two-scale convergence. We have that for a bounded sequence $\left\{ u_{\varepsilon
_{h}}\right\} $ in $L^{2}(\Omega )$ there exists a subsequence such that%
\begin{equation}
\lim_{h\rightarrow \infty }\int_{\Omega }u_{\varepsilon _{h}}\left( x\right)
v\left( x,\frac{x}{\varepsilon _{h}}\right) dx=\int_{\Omega
}\int_{Y}u_{0}\left( x,y\right) v\left( x,y\right) dydx  \label{Twosc norm}
\end{equation}%
for all $v\in L^{2}(\Omega ;C_{\sharp }(Y))$,%
\begin{equation}
\lim_{h\rightarrow \infty }\int_{\Omega }\int_{Y}\mathcal{T}_{\varepsilon
_{h}}\left( u_{\varepsilon _{h}}\right) \left( x,y\right) v\left( x,y\right)
dydx=\int_{\Omega }\int_{Y}u_{0}\left( x,y\right) v\left( x,y\right) dydx
\label{Unfolding norm}
\end{equation}%
for all $v\in L^{2}(\Omega \times Y)$ and

\begin{equation}
\lim_{h\rightarrow \infty }\int_{\Omega }u_{\varepsilon _{h}}\left( x\right) 
\mathcal{U}_{\varepsilon _{h}}\left( v\right) \left( x\right)
dx=\int_{\Omega }\int_{Y}u_{0}\left( x,y\right) v\left( x,y\right) dydx\text{%
,}  \label{Unfolding adj}
\end{equation}%
for all $v\in L^{2}(\Omega \times Y)$, with the same $u_{0}\in L^{2}(\Omega
\times Y)$ in all three cases. Hence, (\ref{Twosc norm}), (\ref{Unfolding
norm}) and (\ref{Unfolding adj}) are equivalent.

\begin{remark}
Note that $\left\{ \tau _{h}\right\} =\left\{ \mathcal{U}_{\varepsilon
_{h}}\right\} $ is defined with respect to $X=L^{2}(\Omega \times Y)$, while
the original definition of two-scale compatibility was based on $X \subset \nobreak L^{2}(\Omega \times \nobreak Y)$.
\end{remark}

Considering the gradient of a bounded sequence $\left\{ u_{\varepsilon_h} \right\}$ in $H^1(\Omega)$, we have the following corollary to Theorem~\ref{Pr adj tvaskalekomp} stating a relation between periodic two-scale convergence, periodic unfolding and general two-scale convergence.

\begin{corollary}
\label{Th: gradsplit}Let $\left\{ u_{\varepsilon _{h}}\right\} $ be a
bounded sequence in $H^{1}(\Omega )$. Then, up to a subsequence,%
\begin{equation}
\nabla u_{\varepsilon _{h}}\left( x\right) \overset{2}{\rightharpoonup }%
\nabla u\left( x\right) +\nabla _{y}u_{1}\left( x,y\right) \text{,}
\label{Gradient two-scale}
\end{equation}%
\begin{equation}
\mathcal{T}_{\varepsilon _{h}}\left( \nabla u_{\varepsilon _{h}}\right)
\left( x,y\right) \rightharpoonup \nabla u\left( x\right) +\nabla
_{y}u_{1}\left( x,y\right) \text{ in\ }L^{2}(\Omega \times Y)^{N}
\label{Gradient unfolding}
\end{equation}%
and $\left\{ \nabla u_{\varepsilon _{h}}\right\} $ two-scale converges with
respect to $\left\{ \mathcal{U}_{\varepsilon _{h}}\right\} $ to $\nabla
u+\nabla _{y}u_{1}$ where $u\in H^{1}(\Omega )$ and $u_{1}\in L^{2}(\Omega
;H_{\sharp }^{1}(Y)/\mathbb{R})$.
\end{corollary}

\begin{proof}
For a proof of (\ref{Gradient two-scale}), see, e.g., Theorem~3 in \cite{Ng1}
or Proposition~1.14 in \cite{All1} and for a proof of (\ref{Gradient
unfolding}) see, e.g., Theorem~3.5 in \cite{CiDaGr3}. Since $\left\{ u_{\varepsilon_h} \right\}$ is bounded in $H^1(\Omega)$, $\left\{ \nabla u_{\varepsilon_{h}} \right\}$ is bounded in $L^{2}(\Omega )^{N}$ and it follows
from Theorem~\ref{Pr adj tvaskalekomp} that $\left\{ \nabla u_{\varepsilon
_{h}}\right\} $ two-scale converges with respect to $\left\{ \mathcal{U}%
_{\varepsilon _{h}}\right\} $ to the same limit as for unfolding and periodic two-scale convergence, i.e., $\nabla u+\nabla _{y}u_{1}$.
\end{proof}

\subsection{Very weak two-scale convergence and unfolding}

We turn to the question whether very weak two-scale convergence has a relation to unfolding analogous to what we saw earlier in this chapter regarding two-scale convergence and unfolding. We recall that the concept of very weak two-scale convergence relies on a special choice of test functions that capture only the micro-oscillations while the macroscopic trend is filtered away. We also recall the convergence \eqref{eq:compvw} in Theorem~\ref{Th: compactness very weak two-scale}, i.e., for $\left\{ u_{\varepsilon
_{h}}\right\} $ bounded in $H_{0}^{1}(\Omega )$, up to a subsequence,
\begin{equation*}
\varepsilon _{h}^{-1}u_{\varepsilon _{h}}\left( x\right) \overset{2}{%
\underset{\operatorname{\mathrm{vw}}}{\rightharpoonup }}u_{1}\left( x,y\right)
\end{equation*}%
where $u_1 \in L^2(\Omega;H^1_{\sharp}(Y)/\mathbb{R})$ is the same as in Corollary~\ref{Th: gradsplit}. The result for unfolding that so far shows the closest relationship to very weak two-scale convergence is \eqref{vw unfolding} in Theorem~\ref{Thunfoldampli}, i.e., for $\left\{ u_{\varepsilon_h} \right\}$ bounded in $H^1(\Omega)$, up to a subsequence,
\begin{equation}
\label{unfolding likt vw}
\varepsilon _{h}^{-1}\left( \mathcal{T}_{\varepsilon _{h}}\left(
u_{\varepsilon _{h}}\right) \left( x,y\right) -\mathcal{M}_{\varepsilon
_{h}}\left( u_{\varepsilon _{h}}\right) \left( x\right) \right)
\rightharpoonup u_{1}\left( x,y\right) +y^{c}\cdot \nabla u\left( x\right)
\end{equation}%
in $L^{2}(\Omega ;H^{1}(Y))$, where $u_{1}$ is the same as above and $u$ is the weak $H^1(\Omega)$-limit of $\left\{ u_{\varepsilon_h} \right\}$. Here the subtraction of $\mathcal{M}_{\varepsilon _{h}}\left( u_{\varepsilon
_{h}}\right) $ plays a role similar to that of the special choice of test functions in very weak two-scale convergence in eliminating the global trend.

Relating very weak two-scale convergence and unfolding results to some form
of expansion may contribute to the understanding for these notions. For the
notion of very weak two-scale convergence, it holds that (see Remark \ref%
{Remark vw})%
\begin{equation}
\varepsilon _{h}^{-1}\left( u_{\varepsilon _{h}}\left( x\right) -u\left(
x\right) \right) \underset{\operatorname{vw}}{\overset{2}{\rightharpoonup }}%
u_{1}\left( x,y\right) \text{.}  \label{vw-two-scale short strengtened}
\end{equation}%
Originally, two-scale convergence was presented as the rigorous
justification of the method of asymptotic expansions. It is easy to see the
relationship between (\ref{vw-two-scale short strengtened}) and the
truncated expansion%
\begin{equation*}
u_{\varepsilon }\left( x\right) \approx u\left( x\right) +\varepsilon
u_{1}\left( x,\frac{x}{\varepsilon }\right) \text{.}
\end{equation*}%
In the setting of (\ref{vw-two-scale short strengtened}) it is reasonable to
consider the alternative that the sequence $\left\{ \varepsilon
_{h}^{-1}\left( u_{\varepsilon _{h}}-u\right) \right\} $ is bounded in $%
L^{2}(\Omega )$. Then, at least up to a subsequence, it will two-scale
converge and hence also converge weakly in $L^{2}(\Omega )$. The complete
two-scale limit may then contain an extra term $w\in L^{2}(\Omega )$, the
weak $L^{2}(\Omega )$-limit not appearing in (\ref{vw-two-scale
short strengtened}), such that%
\begin{equation*}
\varepsilon _{h}^{-1}\left( u_{\varepsilon _{h}}\left( x\right) -u\left(
x\right) \right) \overset{2}{\rightharpoonup }w\left( x\right) +u_{1}\left(
x,y\right) \text{.}
\end{equation*}%
The term $w$ vanishes for the type of test functions used in very weak
two-scale convergence, which is why it does not appear in (\ref{vw-two-scale
short strengtened}). There is, to the knowledge of the authors, no really
obvious interpretation of this term other than simply to be the weak $%
L^{2}(\Omega )$-limit for $\left\{ \varepsilon _{h}^{-1}\left(
u_{\varepsilon _{h}}-u\right) \right\} $, see \cite{FHOLP2}.

Let us compare with the corresponding result (\ref{unfolding likt vw}) for
unfolding, where the extra term is specified as $y^{c}\cdot \nabla u$. From%
\begin{equation*}
\varepsilon ^{-1}\left( \mathcal{T}_{\varepsilon }\left( u_{\varepsilon
}\right) \left( x,y\right) -\mathcal{M}_{\varepsilon }\left( u_{\varepsilon
}\right) \left( x\right) \right) \approx u_{1}\left( x,y\right) +y^{c}\cdot
\nabla u\left( x\right) 
\end{equation*}%
we obtain%
\begin{equation}
\mathcal{T}_{\varepsilon }\left( u_{\varepsilon }\right) \left( x,y\right)
\approx \mathcal{M}_{\varepsilon }\left( u_{\varepsilon }\right) \left(
x\right) +\varepsilon u_{1}\left( x,y\right) +\varepsilon y^{c}\cdot \nabla
u\left( x\right)   \label{Asymt exp unfolding}
\end{equation}%
and we can see how (\ref{Asymt exp unfolding}) reveals the components that $%
\mathcal{T}_{\varepsilon }\left( u_{\varepsilon }\right) $ can be considered
to be built of. For the sake of simplicity we study the one-dimensional case
and fix an interval $Y_{i}^{\varepsilon }$ of length $\varepsilon $ on the $x
$-axis. We let $\mathcal{T}_{\varepsilon }\left( u_{\varepsilon }\right) $
start at the value of $u_{\varepsilon }$ at the beginning of the interval
and follow its development through $Y_{i}^{\varepsilon }$, but let this
happen along a separate $y$-axis. $\mathcal{M}_{\varepsilon }\left(
u_{\varepsilon }\right) $ gives the mean of $\mathcal{T}_{\varepsilon
}\left( u_{\varepsilon }\right) $ during this process and provides us with a
constant base level for $\mathcal{T}_{\varepsilon }\left( u_{\varepsilon
}\right) $ on $Y_{i}^{\varepsilon }$, where changes on both macroscopic and
microscopic level have been removed. Adding $\varepsilon u_{1}$ means that
micro-oscillations along $Y=\left( 0,1\right) ^{N}$ are back and $%
\varepsilon y^{c}\cdot \nabla u$ gives us the slow change while passing
through $Y_{i}^{\varepsilon }$ by means of letting $y$ go across $Y$, which
means that $\varepsilon y^{c}$ passes through $\left( -\frac{\varepsilon }{2}%
,\frac{\varepsilon }{2}\right) $. Multiplying with the slow main tendency $%
\nabla u$ of rate of change we get an approximation of the change in the
macroscopic sense while passing through $Y_{i}^{\varepsilon }$. We may
consider (\ref{Asymt exp unfolding}) as a kind of truncated asymptotic
expansion for the unfolded function $\mathcal{T}_{\varepsilon }\left(
u_{\varepsilon }\right) $. The expressions (\ref{vw-two-scale short
strengtened}) and (\ref{unfolding likt vw}) may now become easier to understand.

Now to the question whether there is a somewhat more explicit relation to unfolding similar to what we saw earlier in this chapter regarding two-scale convergence also for the notion of very weak two-scale convergence, i.e., whether there is an equivalence of the same kind as in (\ref{Twosc norm})--(\ref{Unfolding adj}). It turns out that the amplifying by $\varepsilon _{h}^{-1}$ and the piecewise character of the unfolding operator and its adjoint causes challenges and the existence of an equivalence of the same kind as we found for sequences
bounded in $L^{2}(\Omega )$ and $H^{1}(\Omega )$ is, up to the authors' knowledge, an open question.

Very weak two-scale convergence relies substantially on the quite restrictive
choice of test functions, i.e., $v_{1}\in D(\Omega )$ and $v_{2} \in C_{\sharp}^{\infty} (Y)/\mathbb{R}$. In (\ref{unfolding likt vw}), in contrast to very weak two-scale convergence, there is no special restriction on the class of admissible test functions. In fact, the convergence process is controlled by the entire dual of $L^{2}(\Omega ;H^{1}(Y))$ which is a larger class of functionals than the standard class for unfolding generated by test functions in $L^{2}(\Omega\times Y)$. Assuming that we have $\mathcal{M}_{\varepsilon _{h}}\left( u_{\varepsilon
_{h}}\right) =0$, Remark~\ref{Rem: M0 ger ueps0} gives that (\ref{unfolding likt vw}) reduces to
\begin{equation}
\varepsilon _{h}^{-1}\mathcal{T}_{\varepsilon _{h}}\left( u_{\varepsilon
_{h}}\right) \left( x,y\right) \rightharpoonup u_{1}\left( x,y\right) \text{
in }L^{2}(\Omega ;H^{1}(Y))\text{,} \label{unfolding vw jmf}
\end{equation}%
which is a result more similar to the corresponding result for very weak two-scale convergence. We have introduced a more restrictive class of admissible sequences $\left\{u_{\varepsilon_h}\right\}$ instead of the small class of test functions used in very weak two-scale convergence.

It remains to see if we can enlarge the set of admissible sequences $\left\{
u_{\varepsilon _{h}}\right\} $ if we apply the same smaller class of test
functions used in very weak two-scale convergence also for the concept based
on unfolding. First note that, still for $\mathcal{M}_{\varepsilon _{h}}\left( u_{\varepsilon
_{h}}\right) =0$, by Proposition~\ref{Prop: adjoint} and (\ref{unfolding vw jmf}),
\begin{align*}
\underset{h\rightarrow \infty }{\lim }&\int_{\Omega }\varepsilon
_{h}^{-1}u_{\varepsilon _{h}}\left( x\right) \mathcal{U}_{\varepsilon
_{h}}\left( v_{1}v_{2}\right) \left( x\right) dx \\
&=\underset{h\rightarrow \infty }{\lim }\int_{\Omega }\int\nolimits_{Y}\varepsilon_h ^{-1}\mathcal{T}_{\varepsilon_h
}\left( u_{\varepsilon_h }\right) \left( x,y\right) v\left( x,y\right)
dydx  \\
&=\int_{\Omega
}\int\nolimits_{Y}u_{1}\left( x,y\right) v_{1}\left( x\right) v_{2}\left(
y\right) dydx\text{,}
\end{align*}%
where $v_1 \in D(\Omega)$ and $v_2 \in C_{\sharp}^{\infty}(Y)/\mathbb{R}$. There exist ways to relax the restriction $\mathcal{M}_{\varepsilon _{h}}\left( u_{\varepsilon
_{h}}\right) =0$; in the following theorem we consider a sequence of functions
decomposed into a term of the above kind and a piecewise constant term.

\begin{theorem}
Let
\begin{equation*}
    u_{\varepsilon _{h}}=\widetilde{u}_{\varepsilon _{h}}+\widehat{u}%
    _{\varepsilon _{h}}
\end{equation*}
where $\left\{ \widetilde{u}_{\varepsilon _{h}}\right\} $ is a bounded sequence in $H^{1}(\Omega )$ such that
\begin{equation*}
    \mathcal{M}_{\varepsilon _{h}}\left( \widetilde{u}_{\varepsilon _{h}}\right) = 0
\end{equation*}
and $\left\{ \widehat{u}_{\varepsilon_h} \right\}$ is defined on $\Omega$ by
\begin{equation*}
    \widehat{u}_{\varepsilon_h} \left( x \right) = \left \{
    \begin{aligned}
        & c^h_\xi && \text{for $x \in \varepsilon_h \left( \xi + Y \right)$,} \\
        &0 && \text{for $x \in \Omega \setminus \varepsilon_h \left( \xi + Y \right)$,}
    \end{aligned} \right.
\end{equation*}
for some sequence \smash{$\bigl\{ \bigl\{ c^h_{\xi} \bigr\}_{\xi \in \Xi_{\varepsilon_h}} \bigr\}$}
of lists of real constants. Then, up to a subsequence,
\begin{equation*}
\underset{h\rightarrow \infty }{\lim }\int_{\Omega }\varepsilon
_{h}^{-1}u_{\varepsilon _{h}}\left( x\right) \mathcal{U}_{\varepsilon
_{h}}\left( v_{1}v_{2}\right) \left( x\right) dx=\int_{\Omega
}\int\nolimits_{Y}u_{1}\left( x,y\right) v_{1}\left( x\right) v_{2}\left(
y\right) dydx
\end{equation*}%
for all $v_{1}\in D(\Omega )$, $v_{2}\in C_{\sharp }^{\infty }(Y)/\mathbb{R}$, where $u_{1}$ is the corrector for the two-scale limit of $\left\{ \nabla 
\widetilde{u}_{\varepsilon _{h}}\right\} $.
\end{theorem}

\begin{proof}
We have that%
\begin{multline}
\label{eq:uppdeladintegral}
\int_{\Omega }\varepsilon _{h}^{-1}u_{\varepsilon _{h}}\left( x\right) 
\mathcal{U}_{\varepsilon _{h}}\left( v_{1}v_{2}\right) \left( x\right) dx \\
=\int_{\Omega }\varepsilon _{h}^{-1}\widetilde{u}_{\varepsilon _{h}}\left(
x\right) \mathcal{U}_{\varepsilon _{h}}\left( v_{1}v_{2}\right) \left(
x\right) dx+\int_{\Omega }\varepsilon _{h}^{-1}\widehat{u}_{\varepsilon
_{h}}\left( x\right) \mathcal{U}_{\varepsilon _{h}}\left( v_{1}v_{2}\right)
\left( x\right) dx
\end{multline}%
and we start by showing that the second integral on the right-hand-side equals zero. By the definition of $\mathcal{U}_{\varepsilon_h}$ and Remarks~\ref{rem: y-per function unfolding} and~\ref{rem: adjoint lika med local average} we have
\begin{align*}
    \int_{\Omega }& \varepsilon _{h}^{-1}\widehat{u}_{\varepsilon _{h}}\left(
    x\right) \mathcal{U}_{\varepsilon _{h}}\left( v_{1}v_{2}\right) \left(
    x\right) dx \\
    &=\int_{\Omega }\varepsilon _{h}^{-1}\widehat{u}_{\varepsilon _{h}}\left(
    x\right) \mathcal{M}_{\varepsilon_h} \left( v_1 \right) \left( x \right) v_{2}\left( 
    \frac{x}{\varepsilon _{h}}\right) dx \\
    &=\sum_{\xi \in \Xi_{\varepsilon_h}}\int_{\varepsilon_h \left( \xi + Y \right)}\varepsilon _{h}^{-1}%
    \widehat{u}_{\varepsilon _{h}}\left( x\right) \mathcal{M}_{\varepsilon_h} \left( v_1 \right) \left( x \right) v_{2}\left( \frac{x}{%
    \varepsilon _{h}}\right) dx\text{.}
\end{align*}%
Since $\widehat{u}_{\varepsilon _{h}}$ and
\begin{equation*}
    \mathcal{M}_{\varepsilon _{h}}\left( v_{1}\right) \left( x\right)
    =\int\nolimits_{Y}v_{1}\left( \varepsilon _{h}\left[ \frac{x}{\varepsilon
    _{h}}\right] _{Y}+\varepsilon _{h}z\right) dz
\end{equation*}%
are constant on each $\varepsilon_h$-cube $\varepsilon_h \left( \xi + Y \right)$, and $v_{2}\left( x/\varepsilon _{h}\right) $ has average zero on any such cube we have that
\begin{equation*}
    \sum_{\xi \in \Xi_{\varepsilon_h}}\int_{\varepsilon_h \left( \xi + Y \right)} \varepsilon _{h}^{-1}%
    \widehat{u}_{\varepsilon _{h}}\left( x\right) \mathcal{M}_{\varepsilon _{h}}\left( v_{1}\right) \left( x\right) v_{2}\left( \frac{x}{%
    \varepsilon _{h}}\right) dx = 0
\end{equation*}
and we conclude that the second integral on the right-hand-side in \eqref{eq:uppdeladintegral} equals zero. Hence, letting $h \to \infty$ in \eqref{eq:uppdeladintegral}, applying Proposition~\ref{Prop: adjoint}, Theorem~\ref{Thunfoldampli} and Remark~\ref{Rem: M0 ger ueps0} we arrive at
\begin{align*}
\underset{h\rightarrow \infty }{\lim }
&\int_{\Omega }\varepsilon
_{h}^{-1}u_{\varepsilon _{h}}\left( x\right) \mathcal{U}_{\varepsilon
_{h}}\left( v_{1}v_{2}\right) \left( x\right) dx \\
&= \underset{h\rightarrow
\infty }{\lim }\int_{\Omega }\varepsilon _{h}^{-1}\widetilde{u}_{\varepsilon
_{h}}\left( x\right) \mathcal{U}_{\varepsilon _{h}}\left( v_{1}v_{2}\right)
\left( x\right) dx \\
&=\underset{h\rightarrow \infty }{\lim }\int_{\Omega }\int_{Y}\varepsilon
_{h}^{-1}\mathcal{T}_{\varepsilon _{h}}\left( \widetilde{u}_{\varepsilon
_{h}}\right) \left( x,y\right) v_{1}\left( x\right) v_{2}\left( y\right) dydx \\
&= \int_{\Omega }\int\nolimits_{Y}u_{1}\left( x,y\right) v_{1}\left( x\right)
v_{2}\left( y\right) dydx\text{,}
\end{align*}%
up to a subsequence, for all $v_{1}\in D(\Omega )$ and $v_{2}\in C_{\sharp
}^{\infty }(Y)/\mathbb{R}$ and we have the desired result.
\end{proof}

\begin{remark}
Another result along reminiscent lines was obtained in the early paper on
unfolding in Section~3 in \cite{CiDaGr2}, see also Section~4 in \cite%
{CiDaGr1}, by means of sequences of macro-micro scale-splittings%
\begin{equation*}
u_{\varepsilon _{h}}\left( x\right) =\mathcal{Q}_{\varepsilon _{h}}\left(
u_{\varepsilon _{h}}\right) \left( x\right) +\mathcal{R}_{\varepsilon
_{h}}\left( u_{\varepsilon _{h}}\right) \left( x\right) \text{.}
\end{equation*}%
Here $\mathcal{Q}_{\varepsilon _{h}}\left( u_{\varepsilon _{h}}\right) $
represents the macro tendency for $u_{\varepsilon _{h}}$ obtained by means
of interpolation over the $\varepsilon _{h}$-cells and $\mathcal{R}%
_{\varepsilon _{h}}\left( u_{\varepsilon _{h}}\right) $ the remaining
changes of $u_{\varepsilon _{h}}$ of order $\varepsilon _{h}$. We have that
for $\left\{ u_{\varepsilon _{h}}\right\} $ bounded in $H^{1}(\Omega )$, up
to a subsequence,%
\begin{equation*}
\mathcal{R}_{\varepsilon _{h}}\left( u_{\varepsilon _{h}}\right) \left(
x\right) \rightarrow 0\text{ in }L^{2}(\Omega )
\end{equation*}%
and%
\begin{equation*}
\mathcal{Q}_{\varepsilon _{h}}\left( u_{\varepsilon _{h}}\right) \left(
x\right) \rightharpoonup u\left( x\right) \text{ in }H^{1}(\Omega )\text{;}
\end{equation*}%
see Proposition~4.7~(i) and (ii), respectively,\ in \cite{CiDaGr1}. For
the sequence with amplified oscillations, the mode of convergence is again
quite strong, but only concerns $\mathcal{R}_{\varepsilon _{h}}\left(
u_{\varepsilon _{h}}\right) $, while the macro tendency $\mathcal{Q}%
_{\varepsilon _{h}}\left( u_{\varepsilon _{h}}\right) $ is removed. For $%
\left\{ u_{\varepsilon _{h}}\right\} $ bounded in $H^{1}(\Omega )$ or $%
H_{0}^{1}(\Omega )$,%
\begin{equation*}
\varepsilon _{h}^{-1}\mathcal{T}_{\varepsilon _{h}}\left( \mathcal{R}%
_{\varepsilon _{h}}\left( u_{\varepsilon _{h}}\right) \right) \left(
x,y\right) \rightharpoonup u_{1}^{\prime }\left( x,y\right) \text{ in }%
L^{2}(\Omega ;H^{1}(Y))\text{,}
\end{equation*}%
up to a subsequence where $u_{1}^{\prime }$ is related to $u_{1}$ by%
\begin{equation*}
u_{1}\left( x,y\right) =u_{1}^{\prime }\left( x,y\right)
-\int\nolimits_{Y}u_{1}^{\prime }\left( x,y\right) dy\text{;}
\end{equation*}%
see Proposition~4.8 in \cite{CiDaGr1}.
\end{remark}

To conclude, for sequences studied in this subsection, very weak two-scale convergence and the concepts of convergence based on the periodic unfolding method do not seem to be equivalent in the same immediate sense as for the sequences and concepts of convergence discussed earlier in the
section.


{\small
\begin{thebibliography}{99}
\bibitem{All1} {\it G.~Allaire}:  Homogenization and two-scale convergence. SIAM J. Math. Anal. \textit{23} (1992), no. 6, 1482--1518.
Zbl 0770.35005,
MR 1185639,
DOI 10.1137/0523084

\bibitem{AlBr} {\it G.~Allaire}, {\it M.~Briane}: Multiscale convergence and reiterated homogenization. Proc. Roy. Soc. Edinburgh Sect. A. {\it 126} (1996), 297--342.
Zbl 0866.35017,
MR 1386865,
DOI 10.1017/S0308210500022757

\bibitem{AC} {\it G.~Allaire}, {\it C.~Conca}:  Bloch wave homogenization and spectral asymptotic analysis. J. Math. Pures Appl. {\it 77} (1998), 153--208.
Zbl 0901.35005,
MR 1614641,
DOI 10.1016/S0021-7824(98)80068-8

\bibitem{ACV} {\it G.~Allaire}, {\it C.~Conca}, {\it M.~Vanninathan}: Spectral asymptotics of the Helmholtz model in fluid-solid structures. International J. Numer. Methods Engrg. {\it 46} (1999), 1463--1504.
Zbl 0952.74022,
MR 1728849,
DOI 10.1002/(SICI)1097-0207(19991130)46:9<1463::AID-NME708>3.0.CO;2-9

\bibitem{AAB} {\it M.~Amar}, {\it D.~Andreucci}, {\it D.~Bellaveglia}: The
time-periodic unfolding operator and applications to parabolic
homogenization. Atti Accad. Naz. Lincei Cl. Sci. Fis. Mat. Natur. {\it 28} (2017), 663--700.
Zbl 1383.35015,
MR 3729583,
DOI 10.4171/RLM/781

\bibitem{AAGT} {\it M.~Amar}, {\it D.~Andreucci}, {\it R.~Gianni}, {\it C.~Timofte}: Homogenization results for a class of parabolic equations with a non-local interface condition via time-periodic unfolding. Nonlinear Differ. Equ. Appl. {\it 26} (2019), 28 pp.
Zbl 1435.35035,
MR 4029530,
DOI 10.1007/s00030-019-0592-4

\bibitem{ADH} {\it T.~Arbogast}, {\it J.~Douglas}, {\it U.~Hornung}: Derivation of the double porosity model of single phase flow via homogenization theory. SIAM J. Math. Anal. {\it 21} (1990), 823--836.
Zbl 0698.76106,
MR 1052874,
DOI 10.1137/0521046

\bibitem{BoBr} {\it A.~Boughammoura}, {\it Y.~Braham}: Homogenization of a three-phase composites of double-porosity type. Czechoslovak Math. J. {\it 71} (2021), 45--73.
Zbl 1513.35046,
MR 4226471,
DOI 10.21136/CMJ.2020.0151-19

\bibitem{BLM} {\it A.~Bourgeat}, {\it S.~Luckhaus}, {\it A.~Mikelic}: Convergence of the homogenization process for a double-porosity model of immiscible two-phase flow. SIAM, J. Math. Anal. {\it 27} (1996), 1520--1543.
Zbl 0866.35018,
MR 1416507,
DOI 10.1137/S0036141094276457

\bibitem{CD} {\it J.~Casado-D\'{\i}az}: Two-scale convergence for
nonlinear Dirichlet problems in perforated domains. Proc. Royal Soc.
Edinburgh Sect. A {\it 130} (2000), 249--276.
Zbl 0977.35019,
MR 1750830,
DOI 10.1017/S0308210500000147

\bibitem{CiDaDoGr} {\it D.~Cioranescu}, {\it A.~Damlamian}, {\it P.~Donato}, {\it G.~Griso}, {\it R.~Zaki}: The periodic unfolding method in domains with holes. SIAM J. Math. Anal. {\it 44} (2012), 718--760.
Zbl 1250.49017,
MR 2914248,
DOI 10.1137/100817942

\bibitem{CiDaGr2} {\it D.~Cioranescu}, {\it~A Damlamian}, {\it G.~Griso}: Periodic unfolding and homogenization. C. R. Math. Acad. Sci. Paris {\it 335} (2002), 99--104.
Zbl 1001.49016,
MR 1921004,
DOI 10.1016/S1631-073X(02)02429-9

\bibitem{CiDaGr1} {\it D.~Cioranescu}, {\it A.~Damlamian}, {\it G.~Griso}: The
periodic unfolding method in homogenization. SIAM J. Math. Anal. {\it 40} (2008), 1585--1620.
Zbl 1167.49013,
MR 2466168,
DOI 10.1137/080713148

\bibitem{CiDaGr3} {\it D.~Cioranescu}, {\it A.~Damlamian}, {\it G.~Griso}: The
periodic unfolding method. Theory and applications to partial differential problems. Series in Contemporary Mathematics 3, Springer, Singapore, 2018.
Zbl 1428.35002,
MR 3839345,
DOI 10.1007/978-981-13-3032-2

\bibitem{CiDaGr4} {\it D.~Cioranescu}, {\it A.~Damlamian}, {\it G.~Griso}: The periodic unfolding method for countably many scales. Adv. Math. Sci. Appl. {\it 32} (2023), 17--58.
Zbl 1533.49009,
MR 4683680

\bibitem{CiDoZa} {\it D.~Cioranescu}, {\it P.~Donato}, {\it R.~Zaki}: The periodic unfolding method in perforated domains. Port. Math. {\it 63} (2006), 467--496.
Zbl 1119.49014,
MR 2287278

\bibitem{Flo} {\it L.~Flod\'{e}n}: G-Convergence and homogenization of some sequences of monotone differential operators. Doctoral Thesis 70, Department of Engineering and Sustainable Development, Mid Sweden University, 2009.

\bibitem{FHOPvery-weak} {\it L.~Flod\'{e}n}, {\it A.~Holmbom}, {\it M.~Olsson}, {\it J.~Persson}: Very weak multiscale convergence, Appl. Math. Lett. {\it 23} (2010), 1170--1173.
Zbl 1198.35023,
MR 2665589,
DOI 10.1016/j.aml.2010.05.005

\bibitem{FHOLP2} {\it L.~Flod\'{e}n}, {\it A.~Holmbom}, {\it M.~Olsson~Lindberg}, {\it J.~Persson}: Two-scale convergence. Some remarks and extensions. Pure Appl. Math. Quart. {\it 9} (2013), 461--486.
Zbl 1288.35041,
MR 3138471,
DOI 10.4310/PAMQ.2013.v9.n3.a4

\bibitem{FHOLP33} {\it L.~Flod\'{e}n}, {\it A.~Holmbom}, {\it M.~Olsson~Lindberg}, {\it J.~Persson}: Homogenization of parabolic equations with multiple spatial and temporal scales. In Report 2011:6, preprint series of the Mid Sweden University, 2011.

\bibitem{FHOLParbitrary} {\it L.~Flod\'{e}n}, {\it A.~Holmbom}, {\it M.~Olsson~Lindberg}, {\it J.~Persson}: Homogenization of parabolic equations with an arbitrary number of scales in both space and time. J. Appl. Math. {\it 2014} (2014), 16 pp.
Zbl 1406.35140,
MR 3176810,
DOI 10.1155/2014/101685

\bibitem{Francu2010} {\it J.~Franc\r{u}}: Modification of unfolding approach to two-scaleconvergence. Math. Bohem. {\it 135} (2010), 403--412.
Zbl 1224.35020,
MR 2681014,
DOI 10.21136/MB.2010.140831

\bibitem{Francu2011} {\it J.~Franc\r{u}}: On two-scale convergence and periodic unfolding. Tatra Mt. Math. Publ. {\it 48} (2011), 73--81.
Zbl 1265.35016,
MR 2841107,
DOI 10.2478/v10127-011-0007-3

\bibitem{FrSv2012} {\it J.~Franc\r{u}}, {\it N.~Svanstedt}: Some remarks on two-scale convergence and periodic unfolding. Appl. Math.  {\it 57} (2012), 359--375.
Zbl 1265.35017,
MR 2984608,
DOI 10.1007/s10492-012-0021-z

\bibitem{HOL1} {\it A.~Holmbom}: Some modes of convergence and their application to homogenization and optimal composites design. Doctoral thesis 1996:208D, Department of Mathematics, Lule\aa\ University of Technology, 1996.

\bibitem{HOL} {\it A.~Holmbom}: Homogenization of parabolic equations: an alternative approach and some corrector-type results. Appl. Math. {\it 42} (1997), 321--343.
Zbl 0898.35008,
MR 1467553,
DOI 10.1023/A:1023049608047

\bibitem{HOS} {\it A.~Holmbom}, {\it M.~Olsson}, {\it J.~Silfver}: On the determination of limits for some sequences of functions and differential operators. Can. Appl. Math. Q {\it 14} (2006), 299--329.
Zbl 1135.74039,
MR 2327747

\bibitem{HSSW} {\it A.~Holmbom}, {\it J.~Silfver}, {\it N.~Svanstedt}, {\it N.~Wellander}: On two-scale convergence and related sequential compactness topics. Appl. Math. {\it 51} (2006), 247--262.
Zbl 1164.40304,
MR 2228665,
DOI 10.1007/s10492-006-0014-x

\bibitem{Len} {\it M.~Lenczner}: Homog\'{e}n\'{e}isation d'un circuit 
\'{e}lectrique. C. R. Acad. Sci. Paris S\'{e}rie IIB, {\it 324} (1997), 537--542.
Zbl 0887.35016,
DOI 10.1016/S1251-8069(97)83186-7

\bibitem{Lob} {\it T.~Lobkova}: Homogenization of linear parabolic equations with a certain resonant matching between rapid spatial and temporal oscillations in periodically perforated domains. Acta Math. Appl. Sin. Engl. Ser. {\it 35} (2019), 340--358.
Zbl 1416.35032,
MR 3950176,
DOI 10.1007/s10255-019-0810-1

\bibitem{LNW} {\it D.~Lukkassen}, {\it G.~Nguetseng}, {\it P.~Wall}: Two-scale convergence. Int. J. Pure Appl. Math. {\it 2} (2002), 35--86.
Zbl 1061.35015,
MR 1912819

\bibitem{Nes} {\it S.~Nesenenko}: Homogenization for dislocation based gradient visco-plasticity. J. Math. Anal. Appl. {\it 425} (2015), 133--159.
Zbl 1311.74102,
MR 3299654,
DOI 10.1016/j.jmaa.2014.10.056

\bibitem{Ng1} {\it G.~Nguetseng}: A general convergence result for a functional related to the theory of homogenization. SIAM J. Math. Anal. {\it 20} (1989), 608--623.
Zbl 0688.35007,
MR 0990867,
DOI 10.1137/0520043

\bibitem{Ng2} {\it G.~Nguetseng}: Homogenization structures and applications I. Z. Anal. Anwendungen {\it 22} (2003), 73--108.
Zbl 1045.46031,
MR 1962077,
DOI 10.4171/ZAA/1133

\bibitem{NgSv} {\it G.~Nguetseng}, {\it N.~Svanstedt}: $\Sigma $-convergence. Banach J. Math. Anal. {\it 5} (2011), 101--135.
Zbl 1229.46035,
MR 2738525,
DOI 10.15352/bjma/1313362985

\bibitem{NgWo1} {\it G.~Nguetseng}, {\it J.L.~Woukeng}: $\Sigma$-convergence of nonlinear parabolic operators. Nonlinear Anal. {\it 66} (2007), 968-1004.
Zbl 1116.35011,
MR 2288445,
DOI 10.1016/j.na.2005.12.035

\bibitem{PerPhD} {\it J.~Persson}: Selected Topics in Homogenization.
Doctoral Thesis 127, Department of Engineering and Sustainable Development, Mid Sweden University, 2012.

\bibitem{Sil} {\it J.~Silfver}: G-convergence and homogenization
involving operators compatible with two-scale convergence. Doctoral Thesis 23, Department of Engineering, Physics and Mathematics, Mid Sweden University, 2007.

\bibitem{Zhi} {\it V.V.~Zhikov}: On an extension of the method of
two-scale convergence and its applications. Sb. Math. {\it 191} (2000), 973--1014.
Zbl 0969.35048,
DOI 10.1070/SM2000v191n07ABEH000491
\end{thebibliography}}
\end{document}